\documentclass[11pt,letterpaper, reqno]{amsart}

\newcommand{\R}{\mathbb{R}}

\usepackage{amssymb}
\usepackage{amsthm}
\usepackage{amsxtra}
\usepackage{graphicx}
\usepackage{mathabx}
\usepackage{mathtools}
\usepackage{color}
\usepackage{braket}
\usepackage{physics}
\usepackage{enumerate}
\usepackage{setspace}

\newtheorem{theorem}{Theorem}

\newtheorem{proposition}[theorem]{Proposition}
\newtheorem{lemma}[theorem]{Lemma}
\newtheorem{corollary}[theorem]{Corollary}

\theoremstyle{remark}
\newtheorem{remark}[theorem]{Remark}

\numberwithin{equation}{section}

\numberwithin{theorem}{section}

\numberwithin{table}{section}

\numberwithin{figure}{section}

\ifx\pdfoutput\undefined
  \DeclareGraphicsExtensions{.pstex, .eps}
\else
  \ifx\pdfoutput\relax
    \DeclareGraphicsExtensions{.pstex, .eps}
  \else
    \ifnum\pdfoutput>0
      \DeclareGraphicsExtensions{.pdf}
    \else
      \DeclareGraphicsExtensions{.pstex, .eps}
    \fi
  \fi
\fi

\begin{document}

\title[Semi-classical limit of quantum free energy minimizers for NLH]{Semi-classical limit of quantum free energy minimizers for the gravitational Hartree equation}
\linespread{1.2}

\author{Woocheol Choi}
\address{Department of Mathematics, Sungkyunkwan University, Suwon 16419, Republic of Korea}
\email{choiwc@skku.edu}

\author{Younghun Hong}
\address{Department of Mathematics, Chung-Ang University, Seoul 06974, Republic of Korea}
\email{yhhong@cau.ac.kr}

\author{Jinmyoung Seok}
\address{Department of Mathematics, Kyonggi University, Suwon 16227, Republic of Korea}
\email{jmseok@kgu.ac.kr}



\begin{abstract}

For the gravitational Vlasov-Poisson equation, Guo and Rein constructed a class of classical isotropic states as minimizers of free energies (or energy-Casimir functionals) under mass constraints \cite{GuoRein1}. For the quantum counterpart, that is, the gravitational Hartree equation, isotropic states are constructed as free energy minimizers by Aki, Dolbeault and Sparber \cite{ADS}. In this paper, we are concerned with the correspondence between quantum and classical isotropic states. Precisely, we prove that as the Planck constant $\hbar$ goes to zero, free energy minimizers for the Hartree equation converge to those for the Vlasov-Poisson equation in terms of potential functions as well as via the {Husimi} transform and the T\"oplitz quantization.
\end{abstract}

\maketitle

\section{Introduction}\label{sec: introduction}

\subsection{Background}\label{sec: background}
We consider the gravitational Vlasov-Poisson equation
\begin{equation}\label{gVP}
\partial_t f+p\cdot\nabla_qf+\nabla_q(\tfrac{1}{|\cdot|}*\rho_f)\cdot\nabla_p f=0
\end{equation}
where $f=f(t,q,p):\mathbb{R}\times\mathbb{R}^6\to [0,\infty)$ is a phase-space distribution, and 
$$\rho_f=\int_{\mathbb{R}^3}f(\cdot,p)dp$$
denotes the corresponding density function. This equation describes the mean-field dynamics of a large number of collisionless classical particles interacting each other by their mutual gravitational forces. In astrophysics, it provides a simpler but fairly accurate description of stellar dynamics whose size possibly ranges from $10^8$ stars for dwarfs to $10^{14}$ stars for giants. Here, all pair interactions are approximated by the self-generated potential $-\frac{1}{|\cdot|}*\rho_f$.

The initial-value problem for \eqref{gVP} possesses a unique global classical solution \cite{Pffaffelmoser, Schaeffer, Horst, LioPer}. The nonlinear evolution preserves the mass
$$\mathcal{M}(f):=\int_{\mathbb{R}^6}f(q,p)dqdp$$
and the energy
$$\mathcal{E}(f):=\frac{1}{2}\int_{\mathbb{R}^6}|p|^2f(q,p)dqdp-\frac{1}{2}\int_{\mathbb{R}^6}\frac{\rho_f(x)\rho_f(y)}{|x-y|}dxdy.$$
Moreover, the measure preserving property of the characteristics derives conservation of a Casimir functional
$$\mathcal{C}(f):=\int_{\mathbb{R}^6}\beta\big(f(q,p)\big)dqdp$$
with any given non-negative function $\beta:[0,\infty)\to[0,\infty)$ (see \cite[Chapter 4]{Glassey}). Combining the energy and a Casimir functional, the \textit{energy-Casimir functional}  is defined by
$$\mathcal{J}(f):=\mathcal{E}(f)+\mathcal{C}(f).$$
The quantity $-\mathcal{C}(f)$ also can be considered as a (generalized) entropy, including the standard entropy functional
$$-\int_{\mathbb{R}^6} (f\ln f)(q,p) dqdp,$$
and thus the energy-Casimir functional $\mathcal{J}(f)$ may be called the (generalized) \textit{free energy}.

In many different physical contexts, energy-Casimir/free energy functionals  have played a powerful role to construct stable steady states. As for the Vlasov-Poisson equation \eqref{gVP}, Guo and Rein established a large class of stable isotropic states by the energy-Casimir method \cite{GuoRein1}. This important work is in fact the starting point of our discussion, but we state below a slightly clearer version in the survey article \cite[Theorem 1.1]{Rein}.

We now assume that
\begin{spacing}{1.5}
\noindent \textbf{(A1)} $\beta\in C^1([0,\infty))$ is strictly convex, and $\beta(0)=\beta'(0)=0$,\\
\textbf{(A2)} There exists $m_1>\frac{5}{3}$ such that $\beta(s)\geq C_1s^{m_1}$ for all large $s\geq 0$,\\
\textbf{(A3)} There exists $m_2>\frac{5}{3}$ such that $\beta(s)\leq C_2s^{m_2}$ for all small $s\geq 0$.
\end{spacing}
\noindent For a function $\beta:[0,\infty)\to[0,\infty)$ satisfying \textbf{(A1)}$-$\textbf{(A3)}, we introduce the auxiliary  function $\tilde{\beta}:\mathbb{R}\to\mathbb{R}$ by
\begin{equation}\label{beta tilde}
\tilde{\beta}(s):=\left\{\begin{aligned}
&(\beta')^{-1}(s)&&\textup{for }s\geq0,\\
&0&&\textup{for }s< 0.
\end{aligned}\right.
\end{equation}
The following theorem asserts that steady states can be  constructed as minimizers for the variational problem
\begin{equation}\label{classical VP}
\mathcal{J}_{\min} :=\min_{f\in\mathcal{A}_M}\mathcal{J}(f)
\end{equation}
where $M>0$ and the admissible set $\mathcal{A}_M$ is given by
$$\mathcal{A}_M := \Big\{ f: \mathbb{R}^6\to[0,\infty):\ \mathcal{M}(f)=M\textup{ and }\mathcal{J}(f) < \infty\Big\}.$$
\begin{theorem}[Guo-Rein \cite{GuoRein1, Rein}]\label{classical theorem}
Let $M>0$. If $\beta$ satisfies \textup{\textbf{(A1)}$-$\textbf{(A3)}}, then the following statements hold.
\begin{enumerate}[$(i)$]
\item The minimum free energy $\mathcal{J}_{\min}$ is strictly negative.
\item The minimization problem \eqref{classical VP} possesses a minimizer.
\item A minimizer $\mathcal{Q}$ solves the self-consistent equation
\begin{equation}\label{classical EL}
\mathcal{Q} =\tilde{\beta}(-\mu-E), \quad\textup{with some }\mu>0, 
\end{equation}
where $E=\frac{|p|^2}{2}-U(q)$ is the local total energy with the potential $-U=-\frac{1}{|\cdot|}*\rho_{\mathcal{Q}}$.
\end{enumerate}
\end{theorem}

\begin{remark}\label{Guo-Rein remark}
$(i)$ By the formula \eqref{classical EL}, minimizers in Theorem \ref{classical theorem} solve the Vlasov-Poisson equation \eqref{gVP}. Moreover, they are known to be dynamically stable \cite[Theorem 6.1]{Rein}. See Proposition \ref{prop-Q} for their additional properties.\\
$(ii)$ For instance, $\beta(s)=s^m$ with $m>\frac{5}{3}$ satisfies the assumptions in Theorem \ref{classical theorem}. In this case, $\mathcal{Q}$ is a polytrope $(E_0-E)_+^k$, where $0<k=\frac{1}{m-1}<\frac{3}{2}$.
\end{remark}

Isotropic states of the form \eqref{classical EL} have been received considerable interests from both mathematical and physical communities due to their stability. A famous conjecture asserts orbital stability of any non-increasing radially symmetric steady states for the Vlasov-Poisson equation \cite{BT}. 
Since the breakthrough work of Antonov \cite{Antonov, Antonov2}, their linear stability has been investigated. 
Later, the concentration-compactness principle \cite{Lions1, Lions2} has been applied to establish nonlinear stability for global minimizers of energy-Casimir (respectively, energy) functionals under mass (respectively, mass-Casimir) constraints  \cite{DSS, Guo1, Guo2, GuoRein1, GuoRein2, LMR2, Rein, Schaeffer2, Wolansky}. The nonlinear stability of steady states, not necessarily global minimizers, is firstly studied in \cite{GuoRein4} for the King model. In the celebrated works of Lemou, M\'{e}hats and Rapha\"el \cite{LMR3, LMR4}, the conjecture has finally been settled. We refer to a comprehensive survey article \cite{Mouhot} for the history of the problem.

On the other hand, in some physical situations, quantum effects are needed to be taken in account for gravitational particles. 
The quantum-mechanical counterpart of the Vlasov-Poisson equation \eqref{gVP} is the (non-relativistic) gravitational Hartree equation
\begin{equation}\label{gSP}
i\hbar\partial_t\gamma=\big[-\tfrac{\hbar^2}{2}\Delta-\tfrac{1}{|\cdot|}*\rho_\gamma^\hbar, \gamma\big],
\end{equation}
where $\gamma=\gamma(t)$ is a self-adjoint compact operator acting on $L^2(\mathbb{R}^3)$,
$$\rho_\gamma^\hbar(x):=(2\pi\hbar)^3\gamma(x,x)$$
is the density function of $\gamma$, and $\gamma(x,x')$ denotes the kernel of the integral operator $\gamma$, i.e.,
$$\gamma[\phi](x)=\int_{\mathbb{R}^3}\gamma(x,x')\phi(x')dx'.$$
Here, a small constant $\hbar>0$ represents the Planck constant. By the eigenfunction expansion, the von Neumann equation \eqref{gSP} is equivalent to the system of (possibly infinitely many) coupled equations
$$i\hbar\partial_t\phi_j=-\tfrac{\hbar^2}{2}\Delta\phi_j-(\tfrac{1}{|x|}*\rho^\hbar)\phi_j,\quad j=1, 2, 3, ..., $$
where $\{\phi_j\}_{j=1}^\infty$ is a set of $L^2$-orthonormal functions, and $\rho^\hbar=(2\pi\hbar)^3\sum_{j=1}^\infty\lambda_j|\phi_j|^2$ with $\lambda_j\to0$ (see Section \ref{sec: operator spaces}). This model gives a mean-field description of self-gravitating quantum particles, like boson stars.

For basic well-posedness results for the gravitational Hartree equation \eqref{gSP}, we refer to \cite{BDF1, BDF2, Castella, Chadam, DF}. An important remark is that this quantum model is closely linked to the classical model \eqref{gVP} in the semi-classical limit $\hbar\to0$. In \cite{LP}, Lions and Paul established weak-* convergence from the Hartree to the Vlasov-Poisson equation via the Wigner/Husimi transform. The weak-* convergence can be improved to strong convergence  provided that both the interaction potential and initial data are sufficiently smooth \cite{AKN1, AKN2}. As for fermions, the Vlasov equation is derived from the fermionic linear $N$-body Schr\"odinger equation in the limit $\hbar=\frac{1}{N^{1/3}}\to0$ \cite{GMP, BPSS, BJPSS}.

By the quantum-classical correspondence principle, the quantum model \eqref{gSP} has the analogous conservation laws,
$$\mathcal{M}^\hbar(\gamma):=\textup{Tr}^\hbar\gamma\quad\textup{(mass)},\quad \mathcal{C}^\hbar(\gamma):=\textup{Tr}^\hbar\big(\beta(\gamma)\big)\quad\textup{(Casimir)},$$
and
$$\mathcal{E}^\hbar(\gamma):=\textup{Tr}^\hbar\big((-\tfrac{\hbar^2}{2}\Delta)\gamma\big)-\frac{1}{2}\int_{\mathbb{R}^6}\frac{\rho_\gamma^\hbar(x)\rho_\gamma^\hbar(y)}{|x-y|}dxdy\quad\textup{(energy)},$$
where $\textup{Tr}^\hbar(\cdot)=(2\pi\hbar)^3\textup{Tr}(\cdot)$. The \textit{quantum energy-Casimir functional} (also called the \textit{free energy}) is defined by
$$\mathcal{J}^\hbar(\gamma):=\mathcal{E}^\hbar(\gamma)+\mathcal{C}^\hbar(\gamma).$$
Therefore, one may expect that orbitally stable steady states can be constructed from the analogous quantum-mechanical variational problem
\begin{equation}\label{quantum VP}
\mathcal{J}^\hbar_{\min}:=\min_{\gamma \in \mathcal{A}_M^\hbar}\mathcal{J}^{\hbar}(\gamma)
\end{equation}
with the admissible set
$$\mathcal{A}_M^\hbar := \Big\{ \gamma \in \mathcal{B}(L^2(\R^3)): \gamma\geq0,\text{ compact, self-adjoint, } \mathcal{M}^\hbar(\gamma)=M\textup{ and }\mathcal{J}^\hbar(\gamma) < \infty \Big\}.$$
The following theorem gives an affirmative answer but assuming 
\begin{spacing}{1.5}
\noindent \textbf{(A4)} $\sup_{s > 0}\frac{s\beta'(s)}{\beta(s)} \leq 3$,
\end{spacing}
\noindent instead of \textbf{(A2)} and $\textbf{(A3)}$.

\begin{theorem}[Aki-Dolbeault-Sparber \cite{ADS}]\label{result-ADS}
Let $M>0$. If $\beta$ satisfies \textup{\textbf{(A1)}} and \textup{\textbf{(A4)}}, then the following statements hold for all $\hbar>0$.
\begin{enumerate}[$(i)$]
\item If the minimum free energy $\mathcal{J}_{\min}^\hbar$ is strictly negative, then the quantum minimization problem \eqref{quantum VP} possesses a minimizer.
\item A minimizer $\mathcal{Q}_\hbar$ solves the self-consistent equation
\begin{equation}\label{quantum EL}
\mathcal{Q}_\hbar =\tilde{\beta}(-\mu_\hbar-\hat{E}_\hbar),\quad\textup{with some }\mu_\hbar>0,
\end{equation}
where $\hat{E}_\hbar=-\frac{\hbar^2}{2}\Delta-U_\hbar$ is the quantum Hamiltonian with the mean-field potential $-U_\hbar=-\frac{1}{|x|}*\rho_{\mathcal{Q}_\hbar}^\hbar$.
\end{enumerate}
\end{theorem}

\begin{remark} $(i)$ Theorem \ref{result-ADS} is nothing but a reformulation of Theorem 1.1 in Aki-Dolbeault-Sparber \cite{ADS}, where quantum thermal effects are emphasized. For readers' convenience, we give the precise statement in Appendix \ref{sec: thermal effects}.\\
$(ii)$ Existence of a minimizer is proved under the condition that $\mathcal{J}_{\min}^\hbar<0$. 
This is the case if we take $\beta(s) = T\beta_0(s)$ for fixed $\beta_0$ and sufficiently small $T > 0$.
In case of fixed $T = 1$, it will be shown that a minimizer exists provided that all \textbf{(A1)}-\textbf{(A4)} are satisfied and $\hbar>0$ is small enough (see Theorem \ref{thm1} $(i)$). \\ 
$(iii)$ The minimizer $\mathcal{Q}_\hbar$, given in \eqref{quantum EL}, is a steady state for the gravitational Hartree equation \eqref{gSP}. It includes the polytropes of the form $(E_{\hbar,0}-\hat{E}_\hbar)_+^k$ with $\frac{1}{2}\leq k<\infty$. Note that the range of $k$ is different from the classical case (see Remark \ref{Guo-Rein remark} $(ii)$).\\
$(iv)$ The well-known Cazenave-Lions method \cite{CazLi} applies to see that the set of all minimizers for the problem \eqref{quantum VP} is orbitally stable \cite{ADS}. It however does not directly imply the orbital stability of a minimizer $\mathcal{Q}_\hbar$ itself unless the uniqueness of $\mathcal{Q}_\hbar$, whose proof is still open, is guaranteed.\\
$(iv)$ For the repulsive Hartree equation, a steady state is constructed on a bounded domain \cite{MRW}. For the Hartree-Fock model, existence of free energy minimizer is established for the atomic model \cite{DFL}.
\end{remark}

\subsection{Main result}
A natural question then is to ask whether the above classical and quantum models are consistent as quantum effects (including thermal effects) become negligible. Physically, answering this question would support the quantum model in view of Bohr's correspondence principle. 

In this article, we particularly focus on classical and quantum free energy minimizers under a same mass constraint. From now on, we assume all \textbf{(A1)}-\textbf{(A4)} so that existence of both classical and quantum minimizers are guaranteed. Our main theorem addresses convergence from quantum to classical free energy minimizers in semi-classical limit.

\begin{theorem}[Main theorem]\label{thm1}
Fix $M>0$. Suppose that $\beta$ satisfies \textup{\textbf{(A1)}}-\textup{\textbf{(A4)}}. Then, the following statements hold.
\begin{enumerate}[$(i)$]
\item (Existence of a quantum free energy minimizer) For all sufficiently small $\hbar>0$, the quantum variational problem \eqref{quantum VP} possesses a quantum free energy minimizer of the form 
$$\mathcal{Q}_{\hbar}=\tilde{\beta}\big(-\mu_\hbar-(-\tfrac{\hbar^2}{2}\Delta-U_\hbar)\big).$$
\item (Convergence of the minimum free energy)
$$\lim_{\hbar\to0}\mathcal{J}^{\hbar}_{\min} = \mathcal{J}_{\min}<0.$$
\item (Convergence of the potential function) For any sequence $\{\hbar_n\}_{n=1}^\infty$ such that $\hbar_n\to 0$, there exist a classical free energy minimizer
$$\mathcal{Q}=\tilde{\beta}\big(-\mu-(\tfrac{|p|^2}{2}-U(q))\big)$$
for \eqref{classical VP}, a subsequence of $\{\hbar_n\}_{n=1}^\infty$ $($but still denoted by $\{\hbar_n\}_{n=1}^\infty$$)$ and $\{x_n\}_{n=1}^\infty$ such that $\underset{n\to\infty}\lim \mu_{\hbar_n}=\mu $, and  for any $0\leq\alpha<\frac{1}{5}$, 
$$\lim_{n\to\infty}\|\nabla(U_{\hbar_n}(\cdot-x_n)-U)\|_{L^2(\mathbb{R}^3)}+\|U_{\hbar_n}(\cdot-x_n)-U\|_{C^{0,\alpha}(\mathbb{R}^3)}=0.$$
\item (Convergence via {Husimi} transform and T\"oplitz quantization)
Let {$\tilde{W}_\hbar[\cdot]$} and $\textup{Op}^T_\hbar[\cdot]$ denote the {Husimi} transform and T\"oplitz quantization, respectively (see Appendix A for their definitions). Then,
\[
\lim_{n\to\infty}\|{\tilde{W}_{\hbar_n}}[\tau_{-x_n}\mathcal{Q}_{\hbar_n}\tau_{x_n}] - \mathcal{Q}\|_{L^\infty(\R^6)}  
+\|\tau_{-x_n}\mathcal{Q}_{\hbar_n}\tau_{x_n} - \textup{Op}^T_{\hbar_n}[\mathcal{Q}]\|_{\mathcal{B}(L^2)}=0,
\]
where $\tau_{x}$ is the translation operator given as $f \mapsto f(\cdot-x)$.
\end{enumerate}
\end{theorem}

%\begin{theorem}[Quantum-classical correspondence via Wigner transform]\label{thm2}
%Suppose Assumption \ref{assump} holds.
%Let $\mathcal{Q}_\hbar$ be a family of minimizers for $\mathcal{J}_{\min}^\hbar$. 
%Then there exist a minimizer $\mathcal{Q}$ for $\mathcal{J}_{\min}$ and sequences $\{\hbar_j\} \to 0$, $\{x_j\} \subset \R^3$ such that
%\begin{enumerate}[$(i)$]
%\item 
%\[
%\lim_{j\to\infty}\|W_{\hbar_j}[\tau_{-x_j}\mathcal{Q}_{\hbar_j}\tau_{x_j}] - \mathcal{Q}\|_{L^\infty(\R^6)} = 0.
%\]
%\item 
%\[
%\lim_{j\to\infty}\|\tau_{-x_j}\mathcal{Q}_{\hbar_j}\tau_{x_j} - \textup{Op}^T_{\hbar_j}[\mathcal{Q}]\|_{\mathcal{B}(L^2)} = 0,
%\]
%\end{enumerate}
%\end{theorem}

The three subjects mentioned in Section \ref{sec: background}, namely, stable steady states for the classical stellar dynamics model \eqref{gVP}, those for the quantum model \eqref{gSP}, and semi-classical limit of general states between two models, have been investigated separately by many authors. However, in spite of its physical importance, to the best of authors' knowledge, our main theorem is the first result establishing convergence from quantum to classical steady states in semi-classical limit, which lies in the intersection of the three important subjects.

Due to the conjecture on characterization of stable states for the classical model, an important issue would be optimality of the conditions on $\beta$. However, we do not claim optimality of the assumptions in Theorem \ref{thm1}. Indeed, as in Guo and Rein \cite{GuoRein2}, a different variational formulation, that is, the energy minimization under a mass-Casimir constraint, would be helpful to extend the $\beta$-class. In addition, it could be improved further by the new approach of Lemou, M\'{e}hats and Rapha\"el \cite{LMR4}. However, these approaches have not yet been employed to the quantum model. The currently known result \cite{ADS} still uses the energy-Casimir method, and it forces us to adhere to the free energy minimization problems. On the other hand, we also emphasize that  \textup{\textbf{(A1)}}-\textup{\textbf{(A4)}} in Theorem \ref{thm1} are nothing but the assumptions in Theorem \ref{classical theorem} and \ref{result-ADS}. Therefore, these would be the best possible as long as existence of free energy minimizers is known.

Proving convergence of quantum free energy minimizers is of interests by itself, but we expect that our main theorem will also be helpful to explore their qualitative properties. Indeed, it has been observed in many models that by convergence, uniqueness, non-degeneracy and symmetries of minimizers for a limit model can be transferred to a ``pre-limit" model as long as the given parameter is small or large enough (see \cite{Lenzmann, CHS} for example). We wish to discover  similar phenomena for quantum minimizers, but it is postponed to our forthcoming work.

\subsection{Outline of the proof}

We sketch the main idea and the strategy of the proof of our main theorem. This paper is mostly devoted to the proof of Theorem \ref{thm1} $(ii)$, that is, convergence from the minimum quantum free energy $\mathcal{J}^\hbar_{\min}$ to the minimum classical free energy $\mathcal{J}_{\min}$ as $\hbar\to0$. The rest of the theorem then follows from the concentration-compactness principle in \cite{GuoRein1, Rein} and the properties of the {Husimi} transform and the T\"oplitz quantization.

In order to compare the quantum and the classical minimum free energies, we introduce a family of auxiliary quantum states 
$$\gamma_\hbar:=\tilde{\beta}\big(-\mu-(-\tfrac{\hbar^2}{2}\Delta-U)\big),$$
where $\mathcal{Q} = \tilde{\beta}(-\mu-(\tfrac{|p|^2}{2}-U(q)))$ is a classical free energy minimizer such that $\mathcal{J}_{\min}=\mathcal{J}(\mathcal{Q})$ with $\mathcal{M}(\mathcal{Q})=M$, and a family of auxiliary classical distributions
$$f_\hbar(q,p):=\tilde{\beta}\big(-\mu_{\hbar}-(\tfrac{|p|^2}{2}-U_\hbar(q))\big),$$
where $\mathcal{Q}_\hbar=\tilde{\beta}(-\mu_{\hbar}-(-\tfrac{\hbar^2}{2}\Delta-U_\hbar))$ is a quantum free energy minimizer such that $\mathcal{J}_{\min}^\hbar=\mathcal{J}^\hbar(\mathcal{Q}_\hbar)$ with $\mathcal{M}^\hbar(\mathcal{Q}_\hbar)=M$. Then, we aim to show that 
\begin{equation}\label{key convergence 1}
\lim_{\hbar\to 0}\mathcal{J}^\hbar(\gamma_\hbar)=\mathcal{J}_{\min}\quad\textup{and}\quad \lim_{\hbar\to 0}\mathcal{M}^\hbar(\gamma_\hbar)=M
\end{equation}
and
\begin{equation}\label{key convergence 2}
\lim_{n\to\infty}\left\{\mathcal{J}(f_{\hbar_n})-\mathcal{J}_{\min}^{\hbar_n}\right\}=0\quad\textup{and}\quad \lim_{n\to\infty}\mathcal{M}(f_{\hbar_n})=M
\end{equation}
for any $\{\hbar_n\}_{n=1}^\infty$ with $\hbar_n\searrow 0$. Indeed, since $\mathcal{J}_{\min} < 0$, Theorem \ref{thm1} $(i)$ follows from \eqref{key convergence 1}. Moreover, by definition of the minimum free energies, \eqref{key convergence 1} and \eqref{key convergence 2} immediately imply Theorem \ref{thm1} $(ii)$.

For the proof of \eqref{key convergence 1} and \eqref{key convergence 2}, we employ the two main tools: the T\"oplitz quantization
 and Weyl's law. The T\"oplitz quantization $\text{Op}^T_\hbar[\cdot]$ is a way to convert classical distributions to quantum observables (see Appendix A). There are several different kinds of quantizations, such as the classical quantization and the Weyl quantization. One of the advantage of using the T\"oplitz quantization is that it preserves both positivity and self-adjointness. Our key proposition (Proposition \ref{Toplitz quant convergence}) measures closeness between classical distributions and quantum states
\begin{equation}\label{Toplitz key convergence}
\lim_{\hbar\to0}\Big\|\text{Op}^T_\hbar\Big[\alpha\big(-\mu-(\tfrac{|p|^2}{2}+V(q))\big)\Big] - \alpha\big(-\mu-(-\tfrac{\hbar^2}{2}\Delta+V)\big)\Big\| = 0
\end{equation}
for suitable real-valued function $\alpha: \mathbb{R}\to\mathbb{R}$. Combining with the Lieb-Thirring inequality (Theorem \ref{LT inequality}), it is used to prove convergence from the quantum to the classical potential energies. 

For convergence of the kinetic energy, the mass and the Casimir functionals, we employ Weyl's law (Theorem \ref{Weyl's law} and Proposition \ref{Weyl's law for profile})
$$\lim_{\hbar\to0}(2\pi\hbar)^3N(-E,\hbar, V)=\Big|\Big\{(q,p)\in\mathbb{R}^6:\ \tfrac{|p|^2}{2}+V(q)\leq -E\Big\}\Big|,\quad E>0,$$
where $N(-E,\hbar, V)$ denotes the number of eigenvalues (counting multiplicity) of the Schr\"odinger operator $-\tfrac{\hbar^2}{2}\Delta+V$ less than $-E<0$. Weyl's law provides a precise relation between the number of eigenvalues and the measure in the phase space below a certain energy level. By the layer cake representation, it is used to prove the convergence
\begin{equation}\label{Weyl's law key convergence}
\lim_{\hbar\to0}\textup{Tr}^\hbar\Big(\alpha\big(-\mu-(-\tfrac{\hbar^2}{2}\Delta-U)\big)\Big)=\int_{\mathbb{R}^6}\alpha\big(-\mu-(\tfrac{|p|^2}{2}-U(q))\big)dqdp
\end{equation}
for suitable $\alpha:\mathbb{R}\to\mathbb{R}$. It should be noted that there are a few other ways to prove a similar convergence for the Weyl and the T\"oplitz quantizations exploiting the properties of the transforms (see \cite{AKN1, AKN2} for instance) but it requires some regularity of the function $\alpha$. However, in our setting, both quantum and classical functionals are given in terms of Hamiltonians $-\tfrac{\hbar^2}{2}\Delta-U$ and $\tfrac{|p|^2}{2}-U(q)$, thus we may use Weyl's law. One of benefits of using Weyl's law is that regularity of symbols is not required so that a large class of $\beta$ is included. We also would like to emphasize that Weyl's law, employed in this article, is extended in two ways. First, Weyl's law is typically stated for short-range potentials i.e., $V\in L^{3/2}(\R^3)$ \cite[Theorem XIII.80]{ReedSimon4}, so it cannot be directly applied to the potential functions $U$ and $U_\hbar$. Here, we extend it for long-range potentials $V \in L^r(\R^3)$ with $r \geq 3/2$ (Theorem \ref{Weyl's law}). This is possible, because of the gap from the Lagrange multiplier $-\mu$. Secondly, for the proof of \eqref{key convergence 2}, we need to put $U_\hbar$ in \eqref{Weyl's law key convergence}. To overcome the subtle issue that the potential $U_\hbar$ simultaneously varies in the limit $\hbar\to 0$, we extend Weyl's law for profile decompositions (Proposition \ref{Weyl's law for profile}), and apply it to the profile decomposition of $U_\hbar$.

\subsection{Organization of the rest of the paper}
In Section 2, we define suitable operator spaces for minimizers of \eqref{quantum VP}, and introduce two fundamental inequalities for our analysis, the Lieb-Thirring inequality and the interpolation estimate for distribution functions, which are semi-classically consistent. Section 3 is devoted to the first key tool \eqref{Toplitz key convergence} mentioned in Section 1.3. In Section 4, we prove Weyl's law for long range potentials. Having the two main analytic tools, we prove \eqref{key convergence 1} in Section 5 and \eqref{key convergence 2} in Section 6 and 7. Finally, in Section 8, we complete the proof of Theorem \ref{thm1}. The remaining sections consist of two appendices. In Appendix A, we arrange the properties of the {Husimi} transform and T\"oplitz quantization.
In Appendix B, we state the main result of Aki, Dolbeault and Sparber \cite{ADS} on thermal effects of quantum model, and give some comments.

\subsection{Acknowledgement}
This research of the first author was supported by Basic Science Research Program through the National Research Foundation of Korea (NRF) funded by the Ministry of Education (NRF- 2017R1C1B5076348). This research of the second author was supported by Basic Science Research Program through the National Research Foundation of Korea(NRF) funded by the Ministry of Science and ICT (NRF-2020R1A2C4002615). This research of the third author was supported by Basic Science Research Program through the National Research Foundation of Korea(NRF) funded by the Ministry of Science and ICT (NRF-2020R1C1C1A01006415).
The authors would like to thank Dr. Laurent Lafleche for pointing out an error in using the Wigner transform.

\section{Preliminaries}

We introduce the operator spaces and provide two fundamental inequalities, namely the Lieb-Thirring inequality (Theorem \ref{LT inequality}) and the interpolation estimates (Theorem \ref{interpolation inequality}). They will be crucially used to formulate the quantum and the classical variational problems \eqref{classical VP} and \eqref{quantum VP} in a semi-classically consistent way. In Remark \ref{LT remark}, we emphasize the substantial role of the Lieb-Thirring inequality in our analysis.

\subsection{Operator spaces}\label{sec: operator spaces}
For $1\leq \alpha\leq \infty$, the \emph{Schatten class} $\mathfrak{S}^\alpha$ is defined by the Banach space of compact self-adjoint\footnote{In fact, the Schatten class is defined as a space of compact operators without restricting to self-adjoint operators. However, with abuse of notations, we call its self-adjoint subspace the Schatten class, because only self-adjoint operators are treated in this article.} operators on $L^2=L^2(\mathbb{R}^3;\mathbb{C})$ with the norm
$$\|\gamma\|_{\mathfrak{S}^\alpha}:=\left\{\begin{aligned}
&\big(\textup{Tr} |\gamma|^\alpha\big)^{\frac{1}{\alpha}}&&\textup{if }1\leq \alpha<\infty,\\
&\|\gamma\|   &&\textup{if }\alpha=\infty,
\end{aligned}\right.$$
where we mean by $\|\gamma\|$ the operator norm of $\gamma$. 
In particular, $ \mathfrak{S}^1$ is the trace-class, $\mathfrak{S}^2 $ is the Hilbert-Schmidt class, and $\mathfrak{S}^\infty$ is the set of all compact self-adjoint operators. See the monograph \cite{Simon} for an exhaustive reference. Next, we define the \textit{(homogeneous) trace-class Sobolev space} $\dot{\mathfrak{H}}^1:=\{\gamma: \sqrt{-\Delta}\gamma\sqrt{-\Delta}\in\mathfrak{S}^1\}$ equipped with the norm 
$$\|\gamma\|_{\dot{\mathfrak{H}}^1}:=\big\|\sqrt{-\Delta}\gamma\sqrt{-\Delta}\big\|_{\mathfrak{S}^1}.$$
For notational convenience, if there is no confusion, we denote
\begin{equation}\label{kinetic energy notation}
\textup{Tr}\big((-\Delta)\gamma\big):=\textup{Tr}\big(\sqrt{-\Delta}\gamma\sqrt{-\Delta}\big).
\end{equation}
Indeed, by cyclicity of trace, \eqref{kinetic energy notation} holds for smooth finite-rank operators.  

For alternative interpretation of the operator spaces, we recall that any compact self-adjoint operator has an eigenfunction expansion of the form
$$\gamma=\sum_{j=1}^\infty\lambda_j\ket{\phi_j}\bra{\phi_j},$$
where $\{\phi_j\}_{j=1}^\infty$ is the set of mutually $L^2$ orthonormal eigenfunctions of the given operator and $\lambda_j$'s are the corresponding eigenvalues with $\lambda_j\to 0$. Here, with physicists' bra-ket notation, $\ket{\phi}\bra{\phi}$ denotes a one-particle projector, i.e.,
$$\psi(x)\mapsto\Big(\ket{\phi}\bra{\phi}\ket{\psi}\Big)(x) =  \phi(x)\int_{\mathbb{R}^3} \overline{\phi(x')}\psi(x') dx'.$$
By the eigenfunction expansion, the Schatten $\alpha$-norm is simply the $\ell^\alpha$-norm of eigenvalues,
$$\|\gamma\|_{\mathfrak{S}^\alpha}=\left\|\{\lambda_j\}_{j=1}^\infty\right\|_{\ell^\alpha}$$
and the trace-class Sobolev norm is written by
\[
\|\gamma\|_{\dot{\mathfrak{H}}^1} = \sum_{j=1}^\infty\lambda_j\|\nabla \phi_j\|_{L^2}^2.
\]
The operator $\gamma$ is frequently identified with its kernel $\gamma(x,x')$ represented as
\[
\gamma(x,x') = \sum_{j=1}^\infty \lambda_j\phi_j(x)\overline{\phi_j(x')}.
\]
Then, the density function $\rho_\gamma$ is given by
\[
\rho_\gamma(x) = \gamma(x,x) = \sum_{j=1}^\infty \lambda_j|\phi_j(x)|^2.
\]

\subsection{Two fundamental inequalities}

We recall the following variant of the Lieb-Thirring inequality, which we still call the \textit{Lieb-Thirring inequality} in this article.

\begin{theorem}[Lieb-Thirring inequality]\label{LT inequality}
Let $1\leq\alpha\leq\infty$. If $\gamma\in\dot{\mathfrak{H}}^1\cap\mathfrak{S}^\alpha$ such that $\gamma \geq 0$, then 
\begin{equation}\label{general operator interpolation inequality}
\|\rho_\gamma\|_{L^{\frac{5\alpha-3}{3\alpha-1}}(\mathbb{R}^3)}\lesssim \|\gamma\|_{\mathfrak{S}^\alpha}^{\frac{2\alpha}{5\alpha-3}}\|\sqrt{-\Delta}\gamma\sqrt{-\Delta}\|_{\mathfrak{S}^1}^{\frac{3(\alpha-1)}{5\alpha-3}},
\end{equation}
where $\rho_{\gamma}(x)=\gamma(x,x)$ is the density function. As a consequence, if $1\leq\alpha<\infty$ and $\gamma$ is non-negative, then 
\begin{equation}\label{semi-classical LT}
\|\rho_\gamma^\hbar\|_{L^{\frac{5\alpha-3}{3\alpha-1}}(\mathbb{R}^3)}\lesssim \left\{\textup{Tr}^\hbar\big(\gamma^\alpha\big)\right\}^{\frac{2}{5\alpha-3}} \left\{\textup{Tr}^\hbar\big((-\hbar^2\Delta)\gamma\big)\right\}^{\frac{3(\alpha-1)}{5\alpha-3}}
\end{equation}
for $\hbar\in(0,1]$, where the implicit constants are independent of $\hbar$. Moreover, we have
\begin{equation}\label{endpoint semi-classical LT}
\|\rho_\gamma^\hbar\|_{L^{\frac{5}{3}}(\mathbb{R}^3)}\lesssim \|\gamma\|^{\frac{2}{5}} \left\{\textup{Tr}^\hbar\big((-\hbar^2\Delta)\gamma\big)\right\}^{\frac{3}{5}}.
\end{equation}
\end{theorem}

When $\alpha=\infty$, \eqref{general operator interpolation inequality} is ``the" Lieb-Thirring inequality \cite{LiebThirring}. For the proof of the non-endpoint case $\alpha<\infty$, we refer to the appendix of P.-L. Lions and Paul \cite{LP}. The proof involves the dual formulation of the Lieb-Thirring inequality stated below. It will also be used later to obtain uniform bounds for quantum free energy minimizers (see Lemma \ref{improved uniform boundedness}).

\begin{lemma}[Dual formulation of the Lieb-Thirring inequality]\label{LT dual}
Suppose that $V\in L^{\frac{3}{2}+\alpha}(\mathbb{R}^3)$ for some $\alpha>0$. Let $\lambda_1\leq \lambda_2\leq\cdots <0$ be negative eigenvalues of the Schr\"odinger operator $-\Delta+V$. Then,
\begin{equation}\label{eq: LT dual}
\sum_j |\lambda_j|^{\alpha}\lesssim \int_{\mathbb{R}^3}(V_-)^{\frac{3}{2}+\alpha}dx,
\end{equation}
where $a_-=\max\{-a, 0\}$. As a consequence, if $\lambda_{\hbar;1}\leq \lambda_{\hbar;2}\leq\cdots <0$ are negative eigenvalues of the Schr\"odinger operator $-\tfrac{\hbar^2}{2}\Delta+V$ with $0<\hbar\leq1$, then 
\begin{equation}\label{eq: semi-classical LT dual}
(2\pi\hbar)^3\sum_j |\lambda_{\hbar;j}|^{\alpha}\lesssim \int_{\mathbb{R}^3} (V_-)^{\frac{3}{2}+\alpha}dx,
\end{equation}
where the implicit constant is independent of $\hbar$.
\end{lemma}

\begin{proof}
For the proof of \eqref{eq: LT dual}, we refer to Lemma in Appendix of \cite{LP}. For \eqref{eq: semi-classical LT dual}, we note that $\lambda$ is an eigenvalue of $-\tfrac{\hbar^2}{2}\Delta+V$ if and only if it is an eigenvalue of $-\Delta+V(\hbar\cdot)$ because by simple scaling, $(-\tfrac{\hbar^2}{2}\Delta+V)\phi=\lambda\phi$ can be reformulated as $(-\Delta+V(\frac{\hbar}{\sqrt{2}}\cdot))(\phi(\frac{\hbar}{\sqrt{2}}\cdot))=\lambda\phi(\frac{\hbar}{\sqrt{2}}\cdot)$. Thus, \eqref{eq: LT dual} implies \eqref{eq: semi-classical LT dual}.
\end{proof}

We also recall the interpolation estimate for classical distribution functions. 
\begin{theorem}[Interpolation estimate]\label{interpolation inequality}
Let $1\leq\alpha\leq\infty$. If $f\in L^\alpha(\mathbb{R}^6)$ and $|p|^2f\in L^1(\mathbb{R}^6)$, then 
\begin{equation}\label{eq: interpolation inequality}
\|\rho_f\|_{L^{\frac{5\alpha-3}{3\alpha-1}}(\mathbb{R}^3)}\lesssim \|f\|_{L^\alpha(\mathbb{R}^6)}^{\frac{2\alpha}{5\alpha-3}} \||p|^2f\|_{L^1(\mathbb{R}^6)}^{\frac{3(\alpha-1)}{5\alpha-3}}.
\end{equation}
\end{theorem}

\begin{remark}\label{LT remark}
$(i)$ When $\gamma$ is a one-particle projector, i.e., $\gamma=\ket{\phi}\bra{\phi}$ with $\|\phi_j\|_{L^2}=1$, the Lieb-Thirring inequality \eqref{general operator interpolation inequality} deduces the Gagliardo-Nirenberg inequality
$$\|\phi\|_{L^{\frac{2(5\alpha-3)}{3\alpha-1}}(\mathbb{R}^3)}\lesssim \|\phi\|_{L^2(\mathbb{R}^3)}^{\frac{2\alpha}{5\alpha-3}}\|\nabla\phi\|_{L^2(\mathbb{R}^3)}^{\frac{3(\alpha-1)}{5\alpha-3}}=\|\nabla\phi\|_{L^2(\mathbb{R}^3)}^{\frac{3(\alpha-1)}{5\alpha-3}}.$$
When $\gamma=\sum_{j=1}^N\ket{\phi_j}\bra{\phi_j}$ and $\{\phi_j\}_{j=1}^N$ is an orthonormal set in $L^2(\mathbb{R}^3)$, the inequality \eqref{general operator interpolation inequality} implies that 
$$\Big\|\sum_{j=1}^N|\phi_j|^2\Big\|_{L^{\frac{5\alpha-3}{3\alpha-1}}(\mathbb{R}^3)}\lesssim N^{\frac{2}{5\alpha-3}}\bigg\{\sum_{j=1}^N\|\nabla\phi_j\|_{L^2(\mathbb{R}^3)}^2\bigg\}^{\frac{3(\alpha-1)}{5\alpha-3}}=O(N^{\frac{3\alpha-1}{5\alpha-3}}).$$
To get to the heart of matter, let us naively estimate the square sum using the triangle and the Gagliardo-Nirenberg inequalities without exploiting cancellation from orthogonality,
$$\Big\|\sum_{j=1}^N|\phi_j|^2\Big\|_{L^{\frac{5\alpha-3}{3\alpha-1}}(\mathbb{R}^3)}\leq \sum_{j=1}^N\|\phi_j\|_{L^{\frac{2(5\alpha-3)}{3\alpha-1}}(\mathbb{R}^3)}^2\lesssim \sum_{j=1}^N\|\nabla\phi_j\|_{L^2(\mathbb{R}^3)}^{\frac{6(\alpha-1)}{5\alpha-3}}=O(N).$$
Comparing the above two inequalities, one can see that the Lieb-Thirring inequality captures a summability gain from orthogonality. Similar gain of summability has been recently discovered for time evolution problems by Frank-Lewin-Lieb-Seiringer \cite{FLLS}, and it has been extended in \cite{FS, BHLNS}.\\
$(ii)$ The Lieb-Thirring inequality \eqref{general operator interpolation inequality} is semi-classically consistent in the sense that its reformulated inequalities \eqref{semi-classical LT} and \eqref{endpoint semi-classical LT} agree with the interpolation estimates \eqref{eq: interpolation inequality}. Indeed, by the properties of the T\"oplitz quantization (see Appendix \ref{appendix: Toplitz}), Theorem \ref{interpolation inequality} can be derived inserting $\textup{Op}_\hbar^T[f]$ into the inequality \eqref{semi-classical LT} and then taking $\hbar\to0$.\\
$(iii)$ The Lieb-Thirring inequality \eqref{general operator interpolation inequality} can be extended to
$$\|\rho_\gamma\|_{L^{\frac{5\alpha-3}{3\alpha-1}}(\mathbb{R}^3)}\lesssim \|\gamma\|_{\mathfrak{S}^{\tilde{\alpha}}}^{\frac{2\alpha}{5\alpha-3}}\|\sqrt{-\Delta}\gamma\sqrt{-\Delta}\|_{\mathfrak{S}^1}^{\frac{3(\alpha-1)}{5\alpha-3}}$$
for $1\leq\tilde{\alpha}\leq\alpha$, using the trivial embedding $\mathfrak{S}^{\tilde{\alpha}}\hookrightarrow\mathfrak{S}^\alpha$. However, if we put $\hbar$ as in \eqref{semi-classical LT}, the above inequality leads to
$$\|\rho_\gamma^\hbar\|_{L^{\frac{5\alpha-3}{3\alpha-1}}(\mathbb{R}^3)}\lesssim \hbar^{-\frac{6(\alpha-\tilde{\alpha})}{(5\alpha-3)\tilde{\alpha}}}\left\{\textup{Tr}^\hbar\big(\gamma^{\tilde{\alpha}}\big)\right\}^{\frac{2\alpha}{(5\alpha-3)\tilde{\alpha}}} \left\{\textup{Tr}^\hbar\big((-\hbar^2\Delta)\gamma\big)\right\}^{\frac{3(\alpha-1)}{5\alpha-3}}$$
for non-negative operators. Note that if $\tilde{\alpha}<\alpha$, the implicit constant blows up as $\hbar\to 0$. Thus, only the inequality \eqref{general operator interpolation inequality}, i.e., the case $\tilde{\alpha}=\alpha$, can be used in semi-classical analysis.\\
$(iv)$ In the work of Aki, Dolbeault and Sparber \cite{ADS} (see Theorem \ref{result-ADS}), the Sobolev inequality
\begin{equation}\label{Sobolev}
\|\rho_\gamma\|_{L^3(\mathbb{R}^3)}\lesssim \big\|\sqrt{-\Delta}\gamma\sqrt{-\Delta}\big\|_{\mathfrak{S}^1}
\end{equation}
(with no gain of summability) is employed to construct a minimizer for the variational problem \eqref{quantum VP}, and it allows to include a larger class of $\beta$ 
%(see Assumption \ref{Q assumption})
for fixed $\hbar=1$. However, the inequality \eqref{Sobolev} is not semi-classically consistent. Indeed, the corresponding inequality for distribution functions, i.e., $\|\rho_f\|_{L^3(\mathbb{R}^3)}\lesssim \||v|^2f\|_{L^1(\mathbb{R}^6)}$, does not hold. Therefore, the $\beta$ class in Theorem \ref{result-ADS} needs to be reduced in our analysis.
\end{remark}

\section{T\"oplitz quantization of functions of a Hamiltonian}

In this section, we consider the \textit{T\"oplitz quantization} (or the \textit{Wick quantization}) of functions of a classical Hamiltonian. By the T\"oplitz quantization, we mean the map from classical distributions to quantum states, given by 
$$\textup{Op}^T_\hbar[f] = \frac{1}{(2\pi\hbar)^3}\int_{\R^6} |\varphi^\hbar_{(q,p)}\rangle\langle \varphi^\hbar_{(q,p)}|f(q,p) dqdp,$$
where
$$\varphi^\hbar_{(q,p)}(x) = \frac{1}{(\pi\hbar)^{3/4}}e^{-\frac{|x-q|^2}{2\hbar}}e^{\frac{ip\cdot x}{\hbar}}$$
is a coherent state. See Appendix \ref{appendix: Toplitz} for the basic properties of the T\"oplitz quantization.  

The following proposition asserts that the T\"oplitz quantization of a function of the classical Hamiltonian can be approximated by the corresponding quantum state, and vice versa. Later, this proposition will be used in several places to compare quantum and classical potential energies.

\begin{proposition}[T\"oplitz quantization for a classical Hamiltonian]\label{Toplitz quant convergence}
Let $\alpha:\mathbb{R}\to\mathbb{R}$ be a non-negative continuous function such that $\alpha(t)=0$ for $t\leq 0$. Suppose that $\mu_\hbar\to\mu>0$ as $\hbar\to 0$, and that a family of real-valued functions $V_\hbar$ satisfies
$$\sup_{\hbar\in(0,1]}\|V_{\hbar}\|_{C^{0,\frac{1}{5}}(\mathbb{R}^3)}<\infty.$$
Then, we have
\begin{equation}\label{density correspondence proof claim}
\lim_{\hbar\to0}\left\|\textup{Op}^T_\hbar\Big[\alpha\big(-\mu_\hbar-(\tfrac{|p|^2}{2}+V_{\hbar}(q))\big)\Big]
-\alpha\big(-\mu_\hbar-(-\tfrac{\hbar^2}{2}\Delta+V_{\hbar})\big)\right\| =0.
\end{equation}
\end{proposition}

For the proof, we employ convergence of Hamiltonians in the resolvent sense.

\begin{lemma}[Resolvent convergence]\label{key lemma for density convergence}
Suppose that
$$\sup_{\hbar\in(0,1]}\|V_\hbar\|_{C^{0,\frac{1}{5}}(\mathbb{R}^3)}<\infty.$$
Then, for $0<\hbar\leq 1$, we have
$$\Big\|\textup{Op}^T_\hbar\Big[\frac{1}{\tfrac{1}{2}|p|^2+V_\hbar(q)\pm i}\Big]-\big(-\tfrac{\hbar^2}{2}\Delta+V_\hbar\pm i\big)^{-1}\Big\|\lesssim \hbar^{\frac{1}{10}}\big(1+\|V_\hbar\|_{C^{0,\frac{1}{5}}(\mathbb{R}^3)}\big).$$
\end{lemma}

\begin{proof}
We consider the bilinear form 
$$\mathcal{B}(\psi_1,\psi_2):=\Big\langle\psi_1\Big|\textup{Op}^T_\hbar\Big[\frac{1}{\tfrac{1}{2}|p|^2+V_\hbar(q)\pm i}\Big]-\big(-\tfrac{\hbar^2}{2}\Delta+V_\hbar\pm i\big)^{-1}\Big|\psi_2\Big\rangle.$$
Using that
$$\frac{1}{(2\pi\hbar)^3}\int_{\mathbb{R}^6}|\varphi_{(q,p)}^\hbar\rangle\langle \varphi_{(q,p)}^\hbar| dqdp$$
is the identity on $L^2(\mathbb{R}^3)$, the two operators in the bilinear form can be combined into a single integral,
$$\begin{aligned}
&\textup{Op}^T_\hbar\Big[\frac{1}{\tfrac{1}{2}|p|^2+V_\hbar(q)\pm i}\Big]-\big(-\tfrac{\hbar^2}{2}\Delta+V_\hbar\pm i\big)^{-1}\\
&=\frac{1}{(2\pi\hbar)^3}\int_{\mathbb{R}^6}|\varphi_{(q,p)}^\hbar\rangle\langle \varphi_{(q,p)}^\hbar| \Big\{\frac{1}{\tfrac{1}{2}|p|^2+V_\hbar(q)\pm i}-\big(-\tfrac{\hbar^2}{2}\Delta+V_\hbar\pm i\big)^{-1}\Big\}dqdp\\
&=\frac{1}{(2\pi\hbar)^3}\int_{\mathbb{R}^6}|\varphi_{(q,p)}^\hbar\rangle\langle \varphi_{(q,p)}^\hbar|\frac{-\tfrac{\hbar^2}{2}\Delta-\tfrac{1}{2}|p|^2+V_\hbar-V_\hbar(q)}{\tfrac{1}{2}|p|^2+V_\hbar(q)\pm i} \big(-\tfrac{\hbar^2}{2}\Delta+V_\hbar\pm i\big)^{-1}dqdp.
\end{aligned}$$
Hence, we have
\begin{equation}\label{op norm conv proof}
\mathcal{B}(\psi_1,\psi_2)=\frac{1}{(2\pi\hbar)^3}\int_{\mathbb{R}^6} \frac{\langle\psi_1|\varphi_{(q,p)}^\hbar\rangle \langle (-\frac{\hbar^2}{2}\Delta-\tfrac{1}{2}|p|^2+V_\hbar-V_\hbar(q))\varphi_{(q,p)}^\hbar|\tilde{\psi}_2\rangle}{\tfrac{1}{2}|p|^2+V_\hbar(q)\pm i}dqdp,
\end{equation}
where
$$\tilde{\psi}_2:=\big(-\tfrac{\hbar^2}{2}\Delta+V_\hbar\pm i\big)^{-1}\psi_2.$$
By the Fourier transform, the first factor becomes
$$\langle \psi_1|\varphi_{(q,p)}^\hbar\rangle=\overline{\int_{\mathbb{R}^3} \psi_1(x)\frac{1}{(\pi\hbar)^{3/4}}e^{-\frac{|x-q|^2}{2\hbar}} e^{-\frac{ip\cdot x}{\hbar}}dx}=\overline{(\mathcal{F}_x \Psi_1)\left(\tfrac{p}{\hbar}\right)},$$
where
$$\Psi_1(q,x):=\frac{1}{(\pi\hbar)^{3/4}}e^{-\frac{|x-q|^2}{2\hbar}}\psi_1(x).$$
For the second factor, a direct calculation yields
$$\begin{aligned}
\langle (-\tfrac{\hbar^2}{2}\Delta-\tfrac{1}{2}|p|^2+V_\hbar-V_\hbar(q))\varphi_{(q,p)}^\hbar|\tilde{\psi}_2\rangle&=-ip\cdot (\mathcal{F}_x\vec{\Psi}_{2;1})(\tfrac{p}{\hbar})+(\mathcal{F}_x\Psi_{2;2})(\tfrac{p}{\hbar})\\
&=(\mathcal{F}_x(\hbar\nabla\cdot\vec{\Psi}_{2;1}))(\tfrac{p}{\hbar})+(\mathcal{F}_x\Psi_{2;2})(\tfrac{p}{\hbar}),
\end{aligned}$$
where
$$\vec{\Psi}_{2;1}(q,x):= -(x-q) \frac{1}{(\pi\hbar)^{3/4}}e^{-\frac{|x-q|^2}{2\hbar}}\tilde{\psi}_2(x)$$
and
$$\Psi_{2;2}(q,x):=\Big\{-\tfrac{|x-q|^2}{2}+\tfrac{3\hbar}{2}+V_\hbar(x)-V_\hbar(q)\Big\}\frac{1}{(\pi\hbar)^{3/4}}e^{-\frac{|x-q|^2}{2\hbar}}\tilde{\psi}_2(x).$$
Collecting all, we write \eqref{op norm conv proof} as
$$\begin{aligned}
\mathcal{B}(\psi_1,\psi_2)&=\frac{1}{(2\pi\hbar)^3}\int_{\mathbb{R}^6}\overline{(\mathcal{F}_x\Psi_1)(q,\tfrac{p}{\hbar})} \frac{(\mathcal{F}_x(\hbar\nabla\cdot \vec{\Psi}_{2;1})+\mathcal{F}_x\Psi_{2;2})(q,\tfrac{p}{\hbar})}{\tfrac{1}{2}|p|^2+V_\hbar(q)\pm i}dqdp\\
&=\frac{1}{(2\pi)^3}\int_{\mathbb{R}^6} \overline{\mathcal{F}_x\Psi_1(q,p)}\frac{\mathcal{F}_x(\hbar\nabla\cdot \vec{\Psi}_{2;1})(q,p)+\mathcal{F}_x\Psi_{2;2}(q,p)}{\tfrac{\hbar^2}{2}|p|^2+V_\hbar(q)\pm i} dpdq\\
&=\int_{\mathbb{R}^6}\overline{\Psi_1(q,x)}\big(-\tfrac{\hbar^2}{2}\Delta+V_\hbar(q)\pm i\big)^{-1}\Big\{\hbar\nabla\cdot \vec{\Psi}_{2;1}+\Psi_{2;2}\Big\}(q,x)dxdq,
\end{aligned}$$
where Parseval's identity is used for the $p$-variable in the last step. Hence, by H\"older's inequality, we obtain
$$|\mathcal{B}(\psi_1,\psi_2)|\leq\|\Psi_1\|_{L_{q,x}^2}\Big\{\big\|\big(-\tfrac{\hbar^2}{2}\Delta+V_\hbar(q)\pm i\big)^{-1}\hbar\nabla\cdot \vec{\Psi}_{2;1}\big\|_{L_{q,x}^2}+\big\|\big(-\tfrac{\hbar^2}{2}\Delta+V_\hbar(q)\pm i\big)^{-1}\Psi_{2;2}\big\|_{L_{q,x}^2}\Big\}.$$
Note that as Fourier multiplier operators,
$$\big\|\big(-\tfrac{\hbar^2}{2}\Delta+V_\hbar(q)\pm i\big)^{-1}\hbar\nabla\big\|\sim\Big\|\frac{\hbar p}{\tfrac{\hbar^2}{2}|p|^2+V_\hbar(q)\pm i}\Big\|_{L_p^\infty}\lesssim1+\|V_\hbar\|_{L^\infty}^{1/2}$$
and
$$\big\|\big(-\tfrac{\hbar^2}{2}\Delta+V_\hbar(q)\pm i\big)^{-1}\big\|\leq 1.$$
Thus, it follows that 
$$|\mathcal{B}(\psi_1,\psi_2)|\lesssim \|\Psi_1\|_{L_{q,x}^2}\left\{(1+\|V_\hbar\|_{L^\infty})\|\vec{\Psi}_{2;1}\|_{L_{q,x}^2}+\left\|\Psi_{2;2}\right\|_{L_{q,x}^2}\right\}.$$
By the definitions of $\Psi_1$, $\vec{\Psi}_{2;1}$ and $\Psi_{2;2}$, we have
$$\begin{aligned}
\|\Psi_1\|_{L_{q,x}^2}&=\|\psi_1\|_{L^2},\\
\|\vec{\Psi}_{2;1}\|_{L_{q,x}^2}&\lesssim \sqrt{\hbar}\big\|\big(-\tfrac{\hbar^2}{2}\Delta+V_\hbar\pm i\big)^{-1}\psi_2\big\|_{L_x^2}\lesssim\sqrt{\hbar}\|\psi_2\|_{L^2}
\end{aligned}$$
and
$$\begin{aligned}
\|\Psi_{2;2}\|_{L_{q,x}^2}&\lesssim \Big\{\hbar+\hbar^{\frac{1}{10}}\|V_\hbar\|_{C^{0,\frac{1}{5}}}\Big\}\big\|\big(-\tfrac{\hbar^2}{2}\Delta+V_\hbar\pm i\big)^{-1}\psi_2\big\|_{L_x^2}\\
&\lesssim\hbar^{\frac{1}{10}} (1+\|V_\hbar\|_{C^{0,\frac{1}{5}}})\|\psi_2\|_{L^2}.
\end{aligned}$$
Inserting them into the above inequality, we prove that 
$$|\mathcal{B}(\psi_1,\psi_2)|\lesssim \hbar^{\frac{1}{10}}\big(1+\|V_\hbar\|_{C^{0,\frac{1}{5}}}\big)\|\psi_1\|_{L^2}\|\psi_2\|_{L^2}.$$
Therefore, the proposition follows by duality.
\end{proof}

\begin{proof}[Proof of Proposition \ref{Toplitz quant convergence}]
For notational convenience, we denote
$$H_\hbar= \tfrac{1}{2}|p|^2+V_\hbar(q)\quad\textup{and}\quad \hat{H}_\hbar=-\tfrac{\hbar^2}{2}\Delta+V_\hbar.$$
Observe that
$$-\mu_\hbar-H_\hbar=-V_\hbar(q)-(\mu_{\hbar}+\tfrac{1}{2}|p|^2)\leq \|V_\hbar\|_{L^\infty}$$
and as a quadratic form,
$$-\mu_\hbar-\hat{H}_\hbar=-V_\hbar-\big(\mu_\hbar-\tfrac{\hbar^2}{2}\Delta\big)\leq \|V_\hbar\|_{L^\infty}.$$
Thus, introducing a continuous function $\tilde{\alpha}:\mathbb{R}\to\mathbb{R}$ such that
$$\tilde{\alpha}(t)=\left\{\begin{aligned}
&\alpha(t) &&\textup{if }t\leq A+1, \\
&0 &&\textup{if }t\geq A+2,
\end{aligned}\right.$$
where
$$A:=\sup_{\hbar\in(0,1]}\|V_\hbar\|_{L^\infty},$$
we may replace $\alpha$ by $\tilde{\alpha}$ in \eqref{density correspondence proof claim}. 

Recall that by the Stone-Weierstrass theorem, polynomials in $\frac{1}{t+i}$ and $\frac{1}{t-i}$ are dense in $C_0(\mathbb{R})$, that is, the collection of continuous functions vanishing at infinity (see \cite{ReedSimon1}). Hence, given $\epsilon>0$, there exists a polynomial
\[
P(t) = \sum_{j,k=1}^N a_{jk}\frac{1}{(t+i)^j}\frac{1}{(t-i)^k}
\]
such that
$$\|\tilde{\alpha}-P\|_{C_0(\mathbb{R};\mathbb{C})}\leq\epsilon.$$
Using this polynomial, we decompose 
\[
\begin{aligned}
&\textup{Op}^T_\hbar\big[\alpha(-\mu_\hbar-H_\hbar)_+\big]-\alpha(-\mu_\hbar-\hat{H}_\hbar)_+\\
& =  \textup{Op}^T_\hbar\big[\tilde{\alpha}(-\mu_\hbar-H_\hbar)\big]- \tilde{\alpha}(-\mu_\hbar-\hat{H}_\hbar) \\
& =\underbrace{\textup{Op}^T_\hbar\big[P(-\mu_\hbar-H_\hbar))\big]-P(-\mu_\hbar-\hat{H}_\hbar)}_{A_\hbar}+ \underbrace{\textup{Op}^T_\hbar\big[(\tilde{\alpha}-P)(-\mu_\hbar-H_\hbar)\big]}_{B_\hbar}  \\
&\quad -\underbrace{(\tilde{\alpha}-P)(-\mu_\hbar-\hat{H}_\hbar)}_{C_\hbar}.
\end{aligned}
\]
By the property of the T\"oplitz quantization (Propositions \ref{TQ-bounded}) and functional calculus (see \cite{ReedSimon1} for instance) respectively, we obtain 
$$\begin{aligned}
\|B_\hbar\| &\leq \big\|(\tilde{\alpha}-P)(-\mu_\hbar-H_\hbar)\big\|_{L^\infty(\mathbb{R}^6)} \leq \epsilon,\\
\|C_\hbar\|&\leq \|\tilde{\alpha}-P\|_{L^\infty(\mathbb{R})}\leq\epsilon.
\end{aligned}$$
For $A_\hbar$, we further decompose as 
$$A_\hbar = A_{1;\hbar}+A_{2;\hbar},$$
where
\[
\begin{aligned}
A_{1;\hbar}&= \sum_{j,k=1}^N(-1)^{j+k}a_{jk}\textup{Op}^T_\hbar\left[\frac{1}{(H_\hbar+\mu_\hbar-i)^j (H_\hbar+\mu_\hbar+i)^k}\right]\\
&\quad-\sum_{j,k=1}^N(-1)^{j+k}a_{jk}\textup{Op}^T_\hbar\left[\frac{1}{H_\hbar +\mu_\hbar-i}\right]^j\textup{Op}^T_\hbar\left[\frac{1}{H_\hbar+\mu_\hbar+i}\right]^k
\end{aligned}
\]
and
\[
\begin{aligned}
A_{2;\hbar} &=\sum_{j,k=1}^N(-1)^{j+k}a_{jk}\textup{Op}^T_\hbar\left[\frac{1}{H_\hbar +\mu_\hbar-i}\right]^j\textup{Op}^T_\hbar\left[\frac{1}{H_\hbar+\mu_\hbar+i}\right]^k\\
&\quad- \sum_{j,k=1}^N(-1)^{j+k}a_{jk}\frac{1}{(\hat{H}_\hbar+\mu_\hbar-i)^j}\frac{1}{(\hat{H}_\hbar+\mu_\hbar+i)^k}.
\end{aligned}
\]
Then, Proposition \ref{TQ-est2} implies that $\|A_{1;\hbar}\|\to 0$, while Lemma \ref{key lemma for density convergence} implies that $\|A_{2;\hbar}\|\to 0$. Therefore, collecting all, we complete the proof. 
\end{proof}

\section{Weyl's law for long-range potentials}

In the next section, Weyl's law will be employed to compare quantum and classical functionals. Weyl's law is a classical theorem providing an asymptotic relation between the number of eigenvalues and the volume in the phase space. As for Schr\"odinger operators with decaying potentials, Weyl's law is mostly stated under the assumption that a potential is contained in $L^{\frac{3}{2}}(\mathbb{R}^3)$, and this assumption is unavoidable to count all negative eigenvalues. Meanwhile, the potential functions we deal with here are not in $L^{\frac{3}{2}}(\mathbb{R}^3)$, because $U=\frac{1}{|x|}*\rho_{\mathcal{Q}}$ decay at best $\sim\frac{1}{|x|}$ as $x\to\infty$. Nevertheless, in our case,  the cutoffs in free energy minimizers rule out negative eigenvalues near zero, and this fact allows us to include long-range potentials. In this section, we give a version of Weyl's law that fits into our long range setting. 

Let $V$ be a real-valued potential. For $E>0$ and $\hbar>0$, we denote by $N(-E,\hbar, V)$ the number of eigenvalues (counting multiplicity) of the Schr\"odinger operator 
$-\tfrac{\hbar^2}{2}\Delta+V$ less than $-E$, i.e.,
$$N(-E,\hbar, V)=\textup{dim}\Big\{\textup{Ran}\ \mathbf{1}_{(-\infty,-E)}(-\tfrac{\hbar^2}{2}\Delta+V)\Big\}.$$

\begin{theorem}[Weyl's law]\label{Weyl's law}
Let $E>0$. Suppose that $V\in L^r(\mathbb{R}^3)$ for some $r\geq\frac{3}{2}$. Then, we have
\begin{equation}\label{eq: Weyl's law}
\lim_{\hbar\to0}(2\pi\hbar)^3N(-E,\hbar, V)=\Big|\Big\{(q,p)\in\mathbb{R}^6:\ \tfrac{|p|^2}{2}+V(q)\leq-E\Big\}\Big|.
\end{equation}
\end{theorem}

\begin{remark}
Since $E>0$, the integral on the right hand side of \eqref{eq: Weyl's law} is finite when $V\in L^r(\mathbb{R}^3)$ with $r\geq\frac{3}{2}$ (see Lemma \ref{classical CLR}).
\end{remark}

In fact, the proof of Theorem \ref{Weyl's law} requires merely minor modifications of that of \cite[Theorem XIII.80]{ReedSimon4}, so we just give a sketch of it. One difference is that the following variant of the Cwikel-Lieb-Rozenblum (CLR) bound is employed to deal with long-range potentials.

\begin{lemma}[Cwikel-Lieb-Rozenblum type bound]\label{CLR bound}
Let $E>0$. Suppose that $V\in L^r(\mathbb{R}^3)$ for some $r\geq\frac{3}{2}$. Then, we have
$$N(-E, \hbar, V)\lesssim \frac{1}{E^{r-\frac{3}{2}}\hbar^3}\|V_-\|_{L^r(\mathbb{R}^3)}^r.$$
\end{lemma}

\begin{proof}
The proof is identical to that in \cite{Frank}, but it is included for readers' convenience.

It suffices to show the inequality for $E=1$ and $\hbar=\sqrt{2}$. Indeed, a simple calculation shows $(-\tfrac{\hbar^2}{2}\Delta+V)\phi=\lambda\phi$ with $\lambda<-E$ if and only if $(-\Delta+\tfrac{1}{E}V(\tfrac{\hbar}{\sqrt{2E}}\cdot))\phi(\tfrac{\hbar}{\sqrt{2E}}\cdot)=\tfrac{\lambda}{E}\phi(\tfrac{\hbar}{\sqrt{2E}}\cdot)$ with $\frac{\lambda}{E}<-1$. Thus, $N(-E,\hbar,V)=N(-1, \sqrt{2}, \tfrac{1}{E}V(\tfrac{\hbar}{\sqrt{2E}}\cdot))$ holds, so the proof of the lemma can be reduced to the special case $E=1$ and $\hbar=\sqrt{2}$.

Let $\psi_1,..., \psi_N$ be linearly independent functions in the spectral space corresponding to the spectrum of $-\Delta +V$ on the negative interval $(-\infty,-1)$. We normalize these functions so that $\bra{\sqrt{1-\Delta}\psi_j}\ket{\sqrt{1-\Delta}\psi_k}=\delta_{jk}$, and set $\gamma=\sum_{j=1}^N\ket{\psi_j}\bra{\psi_j}$. By construction,
$$0\leq\sqrt{1-\Delta}\gamma\sqrt{1-\Delta}\leq 1\textup{ and }\textup{Tr}\gamma^{1/2}(1-\Delta)\gamma^{1/2}=N.$$ Hence, we have
\begin{equation}\label{CLR proof}
\begin{aligned}
0&\geq\textup{Tr}\big(\gamma^{1/2}(1-\Delta+V)\gamma^{1/2}\big)=\textup{Tr}\big(\gamma^{1/2}(1-\Delta)\gamma^{1/2}\big)+\textup{Tr}(V\gamma)\\
&=N+\int_{\mathbb{R}^3} V(x)\rho_\gamma(x)dx\geq N-\int_{\mathbb{R}^3} V_-(x)\rho_\gamma(x)dx\\
&\geq N-\|V_-\|_{L^r}\|\rho_\gamma\|_{L^{r'}}.
\end{aligned}
\end{equation}
On the other hand, by Rumin's inequality \cite{Rumin}
$$\|\rho_\gamma\|_{L^3}^3\lesssim\textup{Tr}\big(\gamma^{1/2}(-\Delta)\gamma^{1/2}\big)$$
for $0\leq\gamma\leq(-\Delta)^{-1}$, we obtain 
$$\begin{aligned}
\|\rho_\gamma\|_{L^{r'}}^{r'}&=\int_{\rho_\gamma(x)\geq 1}+\int_{\rho_\gamma(x)\leq 1} \rho_\gamma(x)^{r'}dx\leq\|\rho_\gamma\|_{L^3}^3+\|\rho_\gamma\|_{L^1}\\
&\lesssim\textup{Tr}\big(\gamma^{1/2}(-\Delta)\gamma^{1/2}\big)+\textup{Tr}\gamma=\textup{Tr}\big(\gamma^{1/2}(1-\Delta)\gamma^{1/2}\big)=N,
\end{aligned}$$
where in the first inequality, we used that $r'\leq 3$. Inserting this bound in \eqref{CLR proof}, we prove the desired inequality.
\end{proof}

We also employ the classical analogue of the above CLR bound.
\begin{lemma}\label{classical CLR}
Let $E>0$. Suppose that $V\in L^r(\mathbb{R}^3)$ for some $r\geq\frac{3}{2}$. Then, we have
$$\Big|\Big\{(q,p)\in\mathbb{R}^6: \tfrac{1}{2}|p|^2+V(q)\leq-E\Big\}\Big|\lesssim E^{-(r-\frac{3}{2})}\|V_-\|_{L^r(\mathbb{R}^3)}^r.$$
\end{lemma}

\begin{proof}
It is obvious that $\tfrac{1}{2}|p|^2+V(q)\leq-1$ if and only if $\frac{1}{2}|p|^2-\tilde{V}(q)\leq-1$, where $\tilde{V}(q):=\max\{-V(q), 1\}$. Hence, a simple estimate yields the lemma with $E=1$,  
$$\Big|\Big\{(q,p): \tfrac{1}{2}|p|^2+V(q)\leq-1\Big\}\Big|=\frac{8\sqrt{2}\pi}{3}\int_{\tilde{V}(x)\geq1} \big(\tilde{V}(x)-1\big)^{\frac{3}{2}}dx\leq\frac{8\sqrt{2}\pi}{3}\int_{\mathbb{R}^3} V_-(x)^r dx.$$
Then, the desired inequality for $E>0$ follows by scaling $p\mapsto\sqrt{E}p$.
\end{proof}

\begin{proof}[Sketch of Proof of Theorem \ref{Weyl's law}]
First, we prove the theorem assuming that $V$ is continuous and it has compact support. We observe that $N(-E,\hbar, V)=N(0,\sqrt{2},\tfrac{1}{\hbar^2}(V+E))$, because $(-\tfrac{\hbar^2}{2}\Delta+V)\phi=\lambda_j\phi$ with $\lambda_j<-E$ if and only if $(-\Delta+\frac{2}{\hbar^2}(V+E))\phi=\tilde{\lambda}_j\phi$ with $\tilde{\lambda}_j=\frac{2}{\hbar^2}(\lambda_j+E)<0$. Thus, we may reformulate \eqref{eq: Weyl's law} in a similar form as \cite[Section XIII.15]{ReedSimon4}:
\begin{equation}\label{eq: Weyl's law'}
\lim_{\hbar\to0}(2\pi\hbar)^3N\big(0,\sqrt{2},\tfrac{2}{\hbar^2}(V+E)\big)=\Big|\Big\{(q,p)\in\mathbb{R}^6:\ \tfrac{1}{2}|p|^2+V(q)\leq-E\Big\}\Big|.
\end{equation}
To show \eqref{eq: Weyl's law'}, we follow the argument in  \cite{ReedSimon4}. We take a large cube $(-k,k)^3$ containing the support of $V$, and decompose into small cubes. Then, we approximate the operator $-\Delta+\frac{2}{\hbar^2}(V+E)$ as a direct sum of $-\Delta+\frac{2E}{h^2}$ on $\mathbb{R}^3\setminus(-k,k)^3$ and Dirichlet and Neumann Laplacians with constant potentials on small cubes. Here, the presence of the constant $E$ is not essential. Thus, \eqref{eq: Weyl's law'} can be proved by the same way.

We now consider the case $V$ is merely in $L^r$. Fix arbitrarily small $\delta\in (0,\frac{E}{2}]$, and for small $\epsilon>0$, let $V_\epsilon\in C_c^\infty(\mathbb{R}^3)$ such that $\|V-V_\epsilon\|_{L^r}\leq\epsilon$. Then, we decompose
$$-\tfrac{\hbar^2}{2}\Delta+V+E=\Big[-\tfrac{(1-\delta)\hbar^2}{2}\Delta+V_\epsilon+E-\delta\Big]+\Big[-\tfrac{\delta\hbar^2}{2}\Delta+\delta+(V-V_\epsilon)\Big].$$
Recalling from the proof of [Reed-Simon, Theorem XIII.80] that
\begin{equation}\label{eigenvalue inequality}
\textup{dim}\Big\{\textup{Ran}\ \mathbf{1}_{(-\infty,0)}(A+B)\Big\}\leq \textup{dim}\Big\{\textup{Ran}\ \mathbf{1}_{(-\infty,0)}(A)\Big\}+\textup{dim}\Big\{\textup{Ran}\ \mathbf{1}_{(-\infty,0)}(B)\Big\}
\end{equation}
provided that the quadratic forms $A$ and $B$ are self-adjoint and bounded below, and $Q(A)\cap Q(B)$ are dense, we write
\begin{equation}\label{L^r Weyl's law - one direction}
\begin{aligned}
N(-E,\hbar,V)&=\textup{dim}\Big\{\textup{Ran}\ \mathbf{1}_{(-\infty,0)}(-\tfrac{\hbar^2}{2}\Delta+V+E)\Big\}\\
&\leq\textup{dim}\Big\{\textup{Ran}\ \mathbf{1}_{(-\infty,0)}(-\tfrac{(1-\delta)\hbar^2}{2}\Delta+V_\epsilon+E-\delta)\Big\}\\
&\quad+\textup{dim}\Big\{\textup{Ran}\ \mathbf{1}_{(-\infty,0)}(-\tfrac{\delta\hbar^2}{2}\Delta+\delta+(V-V_\epsilon))\Big\} \\
&=\textup{dim}\Big\{\textup{Ran}\ \mathbf{1}_{(-\infty,-\frac{E-\delta}{1-\delta})}(-\tfrac{\hbar^2}{2}\Delta+\tfrac{1}{1-\delta}V_\epsilon)\Big\}\\
&\quad+\textup{dim}\Big\{\textup{Ran}\ \mathbf{1}_{(-\infty,-1)}(-\tfrac{\hbar^2}{2}\Delta+\tfrac{1}{\delta}(V-V_\epsilon))\Big\} \\
&=N\Big(-\tfrac{E-\delta}{1-\delta},\hbar,\tfrac{1}{1-\delta}V_\epsilon\Big)+N\Big(-1, \hbar, \tfrac{1}{\delta}(V-V_\epsilon)\Big).
\end{aligned}
\end{equation}
For the first term on the right hand side of \eqref{L^r Weyl's law - one direction}, we apply Theorem \ref{Weyl's law} for compactly supported continuous potentials,
$$\begin{aligned}
\lim_{\hbar\to0}(2\pi\hbar)^3N\Big(-\tfrac{E-\delta}{1-\delta},\hbar,\tfrac{1}{1-\delta}V_\epsilon\Big)&=\Big|\Big\{(q,p)\in\mathbb{R}^6:\ \tfrac{1}{2}|p|^2+\tfrac{1}{1-\delta}V(q)\leq-\tfrac{E-\delta}{1-\delta}\Big\}\Big|\\
&=\frac{1}{(1-\delta)^{\frac{3}{2}}}\Big|\Big\{(q,p)\in\mathbb{R}^6:\ \tfrac{1}{2}|p|^2+V(q)\leq-(E-\delta)\Big\}\Big|.
\end{aligned}.$$
For the second term, we apply the Cwikel-Lieb-Rozenblum bound (Lemma \ref{CLR bound}),
$$(2\pi\hbar)^3N\left(-1, \hbar, \tfrac{1}{\delta}(V-V_\epsilon)\right)\lesssim\frac{1}{\delta^r}\|V-V_\epsilon\|_{L^r}^r=\frac{\epsilon^r}{\delta^r}.$$
Collecting, we obtain
$$\limsup_{n\to\infty}(2\pi\hbar)^3N(-E,\hbar,V)\leq \frac{1}{(1-\delta)^{\frac{3}{2}}}\Big|\Big\{(q,p)\in\mathbb{R}^6:\ \tfrac{1}{2}|p|^2+V(q)\leq-(E-\delta)\Big\}\Big|+O\left(\frac{\epsilon^r}{\delta^r}\right).$$
Therefore, sending $\epsilon\to0$ and then $\delta\to 0$, we prove that 
$$\limsup_{n\to\infty}(2\pi\hbar)^3N(-E,\hbar,V)\leq\Big|\Big\{(q,p)\in\mathbb{R}^6:\ \tfrac{1}{2}|p|^2+V(q)\leq-E\Big\}\Big|.$$
For the opposite direction, we decompose 
$$-\tfrac{\hbar^2}{2}\Delta+V_\epsilon+E+\delta=\Big[-\tfrac{(1-\delta)\hbar^2}{2}\Delta+(V+E)\Big]+\Big[-\tfrac{\delta\hbar^2}{2}\Delta+\delta+(V_\epsilon-V)\Big].$$
Then, repeating the same argument, we prove that
$$\liminf_{\hbar\to 0}(2\pi\hbar)^3N(-E,\hbar,V)\geq\Big|\Big\{(q,p)\in\mathbb{R}^6:\ \tfrac{1}{2}|p|^2+V(q)\leq-E\Big\}\Big|.$$
Combining two inequalities, we complete the proof.
\end{proof}

\section{Quantum states from classical minimizers}\label{sec: CtoQ}

We begin this section with summarizing properties of minimizers for the classical variational problem \eqref{classical VP}.
\begin{proposition}[Properties of classical minimizers \cite{GuoRein1, Rein}]\label{prop-Q}
Suppose that $\beta$ satisfies \textup{\textbf{(A1)}$-$\textbf{(A3)}}. Given $M>0$, let $\mathcal{Q}$ be a minimizer for the variational problem \eqref{classical VP} of the form 
$$\mathcal{Q}=\tilde{\beta}\big(-\mu-(\tfrac{|p|^2}{2}-U(q))\big),$$
where $U=\frac{1}{|x|}*\rho_{\mathcal{Q}}$ (see Theorem \ref{classical theorem}). Then, it satisfies the following properties.
\begin{enumerate}[$(i)$]
\item (Negativity of minimum value) $\mathcal{J}_{\text{min}} = \mathcal{J}(\mathcal{Q}) < 0$.
\item (Support) $\mathcal{Q}$ has compact support in $\R^6$. 
\item (Lagrange multiplier) 
\[
-\mu = \frac{1}{M}\int_{\mathbb{R}^6}(\tfrac{|p|^2}{2}-U(q)+\beta'(\mathcal{Q}))\mathcal{Q} \,dqdp < 0.
\] 
\item (Symmetry) $\rho_{\mathcal{Q}}$ and $U$ are radially symmetric up to translation.
\item (Regularity)
$\rho_{\mathcal{Q}} \in C^1_c(\R^3)$ and $U \in C^2_{\text{loc}}(\R^3) \cap L^{3,\infty}(\R^3)$. In particular, $U\in L^\infty(\R^3)$.
\item (Uniqueness)
If $\beta(s) = s^m,$ for some $m>\frac{5}{3}$, then the minimizer $\mathcal{Q}$ is unique up to translation. 
\end{enumerate}
\end{proposition}

Now, we introduce an auxiliary quantum state by
\begin{equation}\label{QS from classical minimizer}
\gamma_\hbar=\tilde{\beta}\big(-\mu-(-\tfrac{\hbar^2}{2}\Delta-U)\big)
\end{equation}
with the same potential function $-U$ in Proposition \ref{prop-Q}. Indeed, it is a natural quantization of $\mathcal{Q}$ in view of the quantum-classical correspondence $p\leftrightarrow -i\hbar\nabla$. By construction, $\gamma_\hbar$ is nonnegative, compact and self-adjoint.

The main result of this section asserts that the quantum state $\gamma_\hbar$ asymptotically admits the mass constraint, and that its free energy (i.e., the value of the energy-Casimir functional) converges to the classical minimum free energy.

\begin{proposition}[Quantum states from a classical minimizer]\label{prop: CtoQ}
Suppose that $\beta$ satisfies \textup{\textbf{(A1)}$-$\textbf{(A3)}}. The quantum state $\gamma_\hbar$, given by \eqref{QS from classical minimizer}, obeys
$$\lim_{\hbar\to 0}\mathcal{J}^\hbar(\gamma_\hbar)=\mathcal{J}_{\min}\quad\textup{and}\quad \lim_{\hbar\to 0}\mathcal{M}^\hbar(\gamma_\hbar)=M.$$
As a consequence, we have
\begin{equation}\label{min free energy 1}
\limsup_{\hbar\to0}\mathcal{J}_{\min}^\hbar\leq\mathcal{J}_{\min}<0.
\end{equation}
\end{proposition}

The proof of the proposition is divided into two parts. For the first part, using Weyl's law (Theorem \ref{Weyl's law}), we prove convergence of the mass and the Casimir functionals. 

\begin{lemma}\label{CtoQ Casimir}
\begin{equation}\label{CtoQ Casimir application1}
\begin{aligned}
\lim_{\hbar\to0}\mathcal{M}^\hbar(\gamma_\hbar)=M&&\textup{(mass)},\quad\lim_{\hbar\to 0}\mathcal{C}^\hbar(\gamma_\hbar)=\mathcal{C}(\mathcal{Q})&&\textup{(Casimir)},
\end{aligned}
\end{equation}
and
\begin{equation}\label{CtoQ Casimir application2}
\lim_{\hbar\to 0}\textup{Tr}^\hbar\big((-\tfrac{\hbar^2}{2}\Delta-U)\gamma_\hbar\big)=\int_{\mathbb{R}^6}\big(\tfrac{|p|^2}{2}-U(q)\big)\mathcal{Q}(q,p)dqdp.
\end{equation}
\end{lemma}

\begin{proof}
It suffices to show that 
\begin{equation}\label{CtoQ Casimir proof}
\lim_{\hbar\to0}\textup{Tr}^\hbar\Big(\alpha\big(-\mu-(-\tfrac{\hbar^2}{2}\Delta-U)\big)\Big)=\int_{\mathbb{R}^6}\alpha\big(-\mu-(\tfrac{|p|^2}{2}-U(q))\big)dqdp
\end{equation}
for any strictly increasing $\alpha \in C(\mathbb{R})$ such that $\alpha(s) = 0$ for $s\leq 0$. Indeed, if it is proved, we may set $\alpha=\tilde{\beta}$ for the mass, $\alpha=\beta\circ \tilde{\beta}$ for the Casimir functional, or $\alpha=s\tilde{\beta}$ for \eqref{CtoQ Casimir application2}.

For the proof of \eqref{CtoQ Casimir proof}, we use the layer cake representation to write
$$\textup{Tr}\Big(\alpha\big(-\mu-(-\tfrac{\hbar^2}{2}\Delta-U)\big)\Big)=\sum_j\alpha(\lambda_j)=\int_0^\infty \#\{j; \lambda_j>\lambda\} d\nu(\lambda),$$
where $\{\lambda_j\}$ denotes the set of the non-negative eigenvalues of $-\mu-(-\tfrac{\hbar^2}{2}\Delta-U)$, $\# A$ is the total number of elements of a set $A$, and $\nu$ is a measure on the Borel sets of the positive real-line $[0,\infty)$ such that $\alpha(t)=\nu([0,t))$. Note that $(-\mu-(-\tfrac{\hbar^2}{2}\Delta-U))\phi=\lambda_j\phi$ with $\lambda_j>\lambda$ if and only if $(-\tfrac{\hbar^2}{2}\Delta-U)\phi=-(\mu+\lambda_j)\phi$ with $-(\mu+\lambda_j)< -(\mu+\lambda)$. Hence, it follows that
$$\textup{Tr}^\hbar\Big(\alpha\big(-\mu-(-\tfrac{\hbar^2}{2}\Delta-U)\big)\Big)=\int_0^\infty N\big(-(\mu+\lambda),\hbar, -U\big)d\nu(\lambda),$$
where $N(-E,\hbar, V)$ denotes the number of eigenvalues of the Schr\"odinger operator $-\tfrac{\hbar^2}{2}\Delta+V$ less than $-E$. Consequently, Weyl's law (Theorem \ref{Weyl's law}) yields 
$$\begin{aligned}
\lim_{\hbar\to0}\textup{Tr}\Big(\alpha\big(-\mu-(-\tfrac{\hbar^2}{2}\Delta-U)\big)\Big)&=\int_0^\infty \left|\left\{(q,p):\ \tfrac{|p|^2}{2}-U(q)< -(\mu+\lambda) \right\}\right|d\nu(\lambda)\\
&=\int_0^\infty \left|\left\{(q,p):\ -\mu-(\tfrac{|p|^2}{2}-U(q))> \lambda \right\}\right|d\nu(\lambda).
\end{aligned}$$
Therefore, applying the layer cake representation again but backwardly, we obtain the right hand side of \eqref{CtoQ Casimir proof}.
\end{proof}

As a corollary, we have the following bounds.

\begin{corollary}\label{CtoQ Casimir cor}
$$\sup_{\hbar\in(0,1]}\textup{Tr}^\hbar\big((-\hbar^2\Delta)\gamma_\hbar\big)<\infty\quad\textup{and}\quad\sup_{\hbar\in(0,1]}\|\rho_{\gamma_\hbar}^\hbar\|_{L^1(\mathbb{R}^3)\cap L^{\frac{5}{3}}(\mathbb{R}^3)}<\infty.$$
\end{corollary}

\begin{proof}
$\textup{Tr}^\hbar(-\hbar^2\Delta\gamma_\hbar)$ is uniformly bounded, because by \eqref{CtoQ Casimir application1} and \eqref{CtoQ Casimir application2},
$$\begin{aligned}
\textup{Tr}^\hbar\big((-\tfrac{\hbar^2}{2}\Delta)\gamma_\hbar\big)&=\textup{Tr}^\hbar\big((-\tfrac{\hbar^2}{2}\Delta-U)\gamma_\hbar\big)+\textup{Tr}^\hbar(U\gamma_\hbar)\\
&\leq \textup{Tr}^\hbar\big((-\tfrac{\hbar^2}{2}\Delta-U)\gamma_\hbar\big)+\|U\|_{L^\infty}\textup{Tr}^\hbar(\gamma_\hbar)\\
&= \int_{\mathbb{R}^6} \big(\tfrac{|p|^2}{2}-U(q)\big)\mathcal{Q}(q,p)dqdp+\|U\|_{L^\infty}M +o(\hbar).
\end{aligned}$$
The bound for the density function $\rho_{\gamma_\hbar}^\hbar$ follows from the Lieb-Thirring inequality (Theorem \ref{LT inequality}) and the fact that $\|\gamma_\hbar\| \leq \tilde{\beta}(\|U\|_{L^\infty})$.
\end{proof}

Secondly, we employ Proposition \ref{Toplitz quant convergence} to prove convergence of the potential energy.
\begin{lemma}\label{CtoQ potential energy}
\begin{equation}\label{eq: CtoQ potential energy1}
\lim_{\hbar\to0}\int_{\mathbb{R}^6}\frac{\rho_{\gamma_\hbar}^\hbar(x)\rho_{\gamma_\hbar}^\hbar(y)}{|x-y|}dxdy=\int_{\mathbb{R}^6}\frac{\rho_{\mathcal{Q}}(x)\rho_{\mathcal{Q}}(y)}{|x-y|}dxdy
\end{equation}
and
\begin{equation}\label{eq: CtoQ potential energy2}
\lim_{\hbar\to0}\int_{\mathbb{R}^6}\frac{\rho_{\gamma_\hbar}^\hbar(x)\rho_{\mathcal{Q}}(y)}{|x-y|}dxdy=\int_{\mathbb{R}^6}\frac{\rho_{\mathcal{Q}}(x)\rho_{\mathcal{Q}}(y)}{|x-y|}dxdy.
\end{equation}
\end{lemma}

\begin{proof}
For \eqref{eq: CtoQ potential energy1}, we write
\begin{equation}\label{potential energy decomposition}
\begin{aligned}
\int_{\mathbb{R}^6}\frac{\rho_{\gamma_\hbar}^\hbar(x)\rho_{\gamma_\hbar}^\hbar(y)}{|x-y|}-\frac{\rho_{\mathcal{Q}}(x)\rho_{\mathcal{Q}}(y)}{|x-y|}dxdy&=\int_{\mathbb{R}^6}\frac{(\rho_{\gamma_\hbar}^\hbar-\rho_{\mathcal{Q}})(x)(\rho_{\gamma_\hbar}^\hbar+\rho_{\mathcal{Q}})(y)}{|x-y|}dxdy\\
&=\int_{\mathbb{R}^6}\frac{(\rho_{\gamma_\hbar}^\hbar-\rho^\hbar_{\textup{Op}_\hbar^T[\mathcal{Q}]})(x)(\rho_{\gamma_\hbar}^\hbar+\rho_{\mathcal{Q}})(y)}{|x-y|}dxdy\\
&\quad+\int_{\mathbb{R}^6}\frac{(\rho^\hbar_{\textup{Op}_\hbar^T[\mathcal{Q}]}-\rho_{\mathcal{Q}})(x)(\rho_{\gamma_\hbar}^\hbar+\rho_{\mathcal{Q}})(y)}{|x-y|}dxdy\\
&=:I+II.
\end{aligned}
\end{equation}
For $II$, by the Hardy-Littlewood-Sobolev inequality, Proposition \ref{prop-Q} and Corollary \ref{CtoQ Casimir cor}, we obtain 
$$II\lesssim \|\rho^\hbar_{\textup{Op}_\hbar^T[\mathcal{Q}]}-\rho_{\mathcal{Q}}\|_{L^{\frac{6}{5}}}\|\rho_{\gamma_\hbar}^\hbar+\rho_{\mathcal{Q}}\|_{L^\frac{6}{5}}\lesssim \|\rho^\hbar_{\textup{Op}_\hbar^T[\mathcal{Q}]}-\rho_{\mathcal{Q}}\|_{L^{\frac{6}{5}}}.$$
Hence, Proposition \ref{TQ-conserv} implies that $II\to 0$.

For $I$, applying the Hardy-Littlewood-Sobolev inequality and the $L^p$ interpolation theorem, we obtain
$$\begin{aligned}
I&\lesssim \|\rho_{\gamma_\hbar}^\hbar-\rho^\hbar_{\textup{Op}_\hbar^T[\mathcal{Q}]}\|_{L^{\frac{6}{5}}}\|\rho_{\gamma_\hbar}^\hbar+\rho_{\mathcal{Q}}\|_{L^\frac{6}{5}}\\
&\leq \|\rho_{\gamma_\hbar}^\hbar-\rho^\hbar_{\textup{Op}_\hbar^T[\mathcal{Q}]}\|_{L^{\frac{5}{3}}}^{\frac{5}{12}}\|\rho_{\gamma_\hbar}^\hbar-\rho^\hbar_{\textup{Op}_\hbar^T[\mathcal{Q}]}\|_{L^1}^{\frac{7}{12}}\|\rho_{\gamma_\hbar}^\hbar+\rho_{\mathcal{Q}}\|_{L^\frac{6}{5}}\\
&\leq \|\rho_{\gamma_\hbar}^\hbar-\rho^\hbar_{\textup{Op}_\hbar^T[\mathcal{Q}]}\|_{L^{\frac{5}{3}}}^{\frac{5}{12}}\Big\{\|\rho_{\gamma_\hbar}^\hbar\|_{L^1}+\|\rho_{\mathcal{Q}}\|_{L^1}\Big\}^{\frac{7}{12}}\Big\{\|\rho_{\gamma_\hbar}^\hbar\|_{L^\frac{6}{5}}+\|\rho_{\mathcal{Q}}\|_{L^\frac{6}{5}}\Big\}.
\end{aligned}$$
By Proposition \ref{prop-Q} and Corollary \ref{CtoQ Casimir cor}, the last two factors on the right hand side bound are uniformly bounded. Thus, it suffices to show $\rho_{\gamma_\hbar}^\hbar-\rho_{\textup{Op}_\hbar^T[\mathcal{Q}]}^\hbar\to0$ in $L^{\frac{5}{3}}$. Indeed, by the Lieb-Thirring inequality (Theorem \ref{LT inequality}), we have
$$\begin{aligned}
\|\rho_{\gamma_\hbar}^\hbar-\rho_{\textup{Op}^T_\hbar[\mathcal{Q}]}^\hbar\|_{L^{\frac{5}{3}}}
&\lesssim\|\gamma_\hbar-\textup{Op}^T_\hbar[\mathcal{Q}]\|^{\frac{2}{5}}
\Big\{(2\pi\hbar)^3\big\||\nabla|(\gamma_\hbar-\textup{Op}^T_\hbar[\mathcal{Q}])|\nabla|\big\|_{\mathfrak{S}^1}\Big\}^{\frac{3}{5}}\\
&\leq\|\gamma_\hbar-\textup{Op}^T_\hbar[\mathcal{Q}]\|^{\frac{2}{5}}
\Big\{\textup{Tr}^\hbar\big((-\hbar^2\Delta)\gamma_\hbar\big)+\textup{Tr}^\hbar\big((-\hbar^2\Delta)\textup{Op}^T_\hbar[\mathcal{Q}]\big)\Big\}^{\frac{3}{5}}.
\end{aligned}$$
Note that by Corollary \ref{CtoQ Casimir cor}, $\textup{Tr}^\hbar((-\hbar^2\Delta)\gamma_\hbar)$ is uniformly bounded, while by Proposition \ref{TQ-conserv}
$$\textup{Tr}^\hbar\big((-\hbar^2\Delta)\textup{Op}^T_\hbar[\mathcal{Q}]\big)\to\iint |p|^2 \mathcal{Q}(q,p)dqdp.$$
On the other hand, by Proposition \ref{Toplitz quant convergence}, $\gamma_\hbar-\textup{Op}^T_\hbar[\mathcal{Q}]$ converges to zero in the operator norm. Therefore, collecting all, we complete the proof of \eqref{eq: CtoQ potential energy1}. The convergence \eqref{eq: CtoQ potential energy2} can be proved by the same way.
\end{proof}

\begin{proof}[Proof of Proposition \ref{prop: CtoQ}]
By \eqref{CtoQ Casimir application1}, it suffices to show convergence in energy. Indeed, \eqref{CtoQ Casimir application2} and Lemma \ref{CtoQ potential energy} imply that
$$\begin{aligned}
\mathcal{E}^\hbar(\gamma_\hbar)-\mathcal{E}(\mathcal{Q})&=\textup{Tr}^\hbar\big((-\tfrac{\hbar^2}{2}\Delta-U)\gamma_\hbar\big)-\int_{\mathbb{R}^6}\big(\tfrac{|p|^2}{2}-U(q)\big)\mathcal{Q}(q,p)dqdp\\
&\quad+\int_{\mathbb{R}^6} \frac{\rho_{\mathcal{Q}}(x)\rho_{\gamma_\hbar}^\hbar(y)}{|x-y|}dxdy-\int_{\mathbb{R}^6} \frac{\rho_{\mathcal{Q}}(x)\rho_{\mathcal{Q}}(y)}{|x-y|}dxdy\\
&\quad-\frac{1}{2}\int_{\mathbb{R}^6}\frac{\rho_{\gamma_\hbar}^\hbar(x)\rho_{\gamma_\hbar}^\hbar(y)}{|x-y|}dxdy+\frac{1}{2}\int_{\mathbb{R}^6} \frac{\rho_{\mathcal{Q}}(x)\rho_{\mathcal{Q}}(y)}{|x-y|}dxdy
\end{aligned}$$
converges to zero.
\end{proof}

\begin{proof}[Proof of Theorem \ref{thm1} $(i)$]
By Theorem \ref{result-ADS} $(i)$, the strict inequality \eqref{min free energy 1} implies existence of a quantum free energy minimizer.
\end{proof}

\section{Properties of the quantum minimizers}\label{sec: prop of quantum minimizer}

In the next section, we will prove
\begin{equation}\label{opposite min free energy 1}
\liminf_{\hbar\to0}\mathcal{J}_{\min}^\hbar\geq\mathcal{J}_{\min},
\end{equation}
with which \eqref{min free energy 1} implies $\mathcal{J}_{\min}^\hbar\to\mathcal{J}_{\min}$. In this section, as a preparation, we collect useful properties of quantum minimizers that are uniform in all small $\hbar>0$.

Let
\begin{equation}\label{Q equation for property}
\mathcal{Q}_\hbar=\tilde{\beta}\big(-\mu_\hbar-(-\tfrac{\hbar^2}{2}\Delta-U_h)\big)
\end{equation}
be a minimizer for the variational problem \eqref{quantum VP} whose existence is proved in the previous section.
We denote by $-\mu_j=-\mu_{\hbar;j}$ the negative eigenvalues (counting multiplicity) of the Schr\"odinger operator $-\tfrac{\hbar^2}{2}\Delta-U_\hbar$ with non-decreasing order $-\mu_1\leq-\mu_2\leq\cdots<0$, and let $\phi_j=\phi_{\hbar;j}$ be the $L^2$-normalized real-valued eigenfunction corresponding to the eigenvalue $-\mu_j$, i.e.,
\begin{equation}\label{eigenvalue equation}
(-\tfrac{\hbar^2}{2}\Delta-U_\hbar)\phi_j=-\mu_j\phi_j.
\end{equation}
From these notations and by functional calculus, we deduce the eigenfunction expansion for the quantum minimizer 
\begin{equation}\label{eigenfunction expansion of Q}
\mathcal{Q}_\hbar=\sum_{-\mu_j<-\mu_\hbar} \lambda_j\ket{\phi_j}\bra{\phi_j},
\end{equation}
where
\begin{equation}\label{eigenvalues relation}
\lambda_j:=\tilde{\beta}(\mu_j-\mu_\hbar)\quad(\textup{equivalently, }\mu_j=\mu_\hbar+\beta'(\lambda_j)).
\end{equation}

The following uniform bounds are an immediate consequence of the Lieb-Thirring inequality (Theorem \ref{LT inequality}) and the variational character of quantum minimizers. 

\begin{lemma}[Uniform bounds for quantum minimizers]\label{uniform boundedness}
Suppose that $\beta$ satisfies \textup{\textbf{(A1)}$-$\textbf{(A4)}} and let $\mathcal{Q}_\hbar$ be a minimizer for the quantum variational problem \eqref{quantum VP} of the form \eqref{Q equation for property}. Then, the following hold.
\begin{enumerate}[$(i)$]
\item (Minimum free energy bound)
\[\sup_{\hbar\in(0,1]}\mathcal{J}_{\min}^\hbar<\infty.\]
\item (Kinetic energy and Casimir functional bounds)
\[
\sup_{\hbar\in(0,1]}\textup{Tr}^\hbar\big((-\hbar^2\Delta)\mathcal{Q}_\hbar\big)<\infty,\quad\sup_{\hbar\in(0,1]}\textup{Tr}^\hbar\big(\beta(\mathcal{Q}_\hbar)\big)<\infty.
\]
\item (Density function and potential energy bounds)
\[\sup_{\hbar\in(0,1]}\|\rho_{\mathcal{Q}_\hbar}^\hbar\|_{L^{\frac{6}{5}}(\mathbb{R}^3)}<\infty,\quad \sup_{\hbar\in(0,1]}\int_{\mathbb{R}^6}\frac{\rho_{\mathcal{Q}_\hbar}^\hbar(x)\rho_{\mathcal{Q}_\hbar}^\hbar(y)}{|x-y|}dxdy<\infty\]
\end{enumerate}
\end{lemma}

\begin{proof}
First, we note that $(i)$ is proved in Proposition \ref{prop: CtoQ}.
%For $(i)$, we take the operator $\textup{Op}^T_\hbar[f]$ with the Gaussian distribution $f(q,p)=\frac{M}{\pi^3}e^{-|q|^2-|p|^2}$. Then, by the basic properties of the T\"oplitz quantization (see Appendix \ref{appendix: Toplitz}), the quantum state $\textup{Op}^T_\hbar[f]$ satisfies the mass constraint, i.e.,
%$$\mathcal{M}^\hbar\big(\textup{Op}^T_\hbar[f]\big)=\mathcal{M}(f)=M.$$
%Moreover, we have
%$$\begin{aligned}
%\mathcal{J}^\hbar\big(\textup{Op}^T_\hbar[f]\big)&=\frac{1}{2}\textup{Tr}^\hbar\big((-\hbar^2\Delta)\textup{Op}^T_\hbar[f]\big)-\frac{1}{4}\iint\frac{\rho_{\textup{Op}^T_\hbar[f]}(x)\rho_{\textup{Op}^T_\hbar[f]}(y)}{|x-y|}dxdy+\textup{Tr}^\hbar\Big(\beta\big(\textup{Op}^T_\hbar[f]\big)\Big)\\
%&=\frac{1}{2}\iint|p|^2f(q,p)dqdp-\frac{1}{4}\iint\frac{\rho_f(x)\rho_f(y)}{|x-y|}dxdy+\iint\beta\big(f(q,p)\big)dqdp+o_\hbar(1)\\
%&=\mathcal{J}(f)+o_\hbar(1).
%\end{aligned}$$
%By the definition of $\mathcal{J}^\hbar_{\min}$ as a minimum free energy, it implies that
%$$\mathcal{J}^\hbar_{\min}\leq \mathcal{J}^\hbar(\textup{Op}^T_\hbar[f])=\mathcal{J}(f)+o_\hbar(1).$$
For $(ii)$, we estimate the potential energy by the Hardy-Littlewood-Sobolev inequality and the Lieb-Thirring inequality (Theorem \ref{LT inequality}),
\begin{equation}\label{temp potential energy bound}
\int_{\mathbb{R}^6}\frac{\rho^\hbar_{\mathcal{Q}_\hbar}(x)\rho^\hbar_{\mathcal{Q}_\hbar}(y)}{|x-y|}dxdy\lesssim \|\rho_{\mathcal{Q}_\hbar}\|_{L^{\frac{6}{5}}}^2\lesssim \left\{\textup{Tr}^\hbar\big(\mathcal{Q}_\hbar^{\frac{9}{7}}\big)\right\}^{\frac{7}{6}}\left\{\textup{Tr}^\hbar\big((-\hbar^2\Delta)\mathcal{Q}_\hbar\big)\right\}^{\frac{1}{2}}.
\end{equation}
By the assumption \textup{\textbf{(A2)}}, we have
$$\begin{aligned}
\textup{Tr}^\hbar\big(Q_\hbar^{\frac{9}{7}}\big)&=(2\pi\hbar)^3\sum_j\lambda_j^{\frac{9}{7}}=(2\pi\hbar)^3\sum_{\lambda_j< R}\lambda_j^{\frac{9}{7}}+(2\pi\hbar)^3\sum_{\lambda_j\geq R}\lambda_j^{\frac{9}{7}}\\
&\leq R^{\frac{2}{7}}(2\pi\hbar)^3\sum_{\lambda_j< R}\lambda_j+\frac{1}{R^{m-\frac{9}{7}}}(2\pi\hbar)^3\sum_{\lambda_j\geq R}\lambda_j^m\\
&\leq R^{\frac{2}{7}}M+\frac{1}{CR^{m-\frac{9}{7}}}\textup{Tr}^\hbar\big(\beta(\mathcal{Q}_\hbar)\big).
\end{aligned}$$
for any large $R\geq 1$. Taking $R=\max\{(\frac{\textup{Tr}^\hbar(\beta(\mathcal{Q}_\hbar))}{CM})^{\frac{1}{m-1}},1\}$, the above bound can be optimized as 
$$\textup{Tr}^\hbar\big(Q_\hbar^{\frac{9}{7}}\big)\leq M+2 C^{-\frac{2}{7(m-1)}}M^{\frac{7m-9}{7(m-1)}}\Big\{\textup{Tr}^\hbar\big(\beta(\mathcal{Q}_\hbar)\big)\Big\}^{\frac{2}{7(m-1)}}.$$
Hence, inserting it into \eqref{temp potential energy bound} and then applying Young's inequality with $\frac{3m-5}{6(m-1)}+\frac{1}{3(m-1)}+\frac{1}{2}=1$, we prove that there exists $D>0$, independent of $\hbar\in(0,1]$, such that 
\begin{equation}\label{temp potential bound}
\frac{1}{2}\int_{\mathbb{R}^6}\frac{\rho_{\mathcal{Q}_\hbar}^\hbar(x)\rho_{\mathcal{Q}_\hbar}^\hbar(y)}{|x-y|}dxdy\leq D M^{\frac{7m-9}{3m-5}}+\frac{3}{4}\textup{Tr}^\hbar\beta(\mathcal{Q}_\hbar)+\frac{1}{4}\textup{Tr}^\hbar\big((-\hbar^2\Delta)\mathcal{Q}_\hbar\big).
\end{equation}
Consequently, a lower bound is obtained for the energy-Casimir functional,  
$$\begin{aligned}
\mathcal{J}_{\min}^\hbar&=\frac{1}{2}\textup{Tr}^\hbar\big((-\hbar^2\Delta)\mathcal{Q}_\hbar\big)-\frac{1}{2}\int_{\mathbb{R}^6}\frac{\rho_{\mathcal{Q}_\hbar}^\hbar(x)\rho_{\mathcal{Q}_\hbar}^\hbar(y)}{|x-y|}dxdy+\textup{Tr}^\hbar\beta(\mathcal{Q}_\hbar)\\
&\geq \frac{1}{4}\left\{\textup{Tr}^\hbar\big((-\hbar^2\Delta)\mathcal{Q}_\hbar\big)+\textup{Tr}^\hbar\big(\beta(\mathcal{Q}_\hbar)\big)\right\}-D M^{\frac{7m-9}{3m-5}}.
\end{aligned}$$
Since $\mathcal{J}_{\min}^\hbar$ is uniformly bounded, this completes the proof of $(ii)$. 
Then, the proof of $(iii)$ follows from the Lieb-Thirring inequality (Theorem \ref{LT inequality}) with $(ii)$ and the Hardy-Littlewood-Sobolev inequality. 
\end{proof}

Next, exploiting the Euler-Lagrange equation \eqref{Q equation for property} little bit further, we improve some of the previous bounds. 

\begin{lemma}[More uniform bounds for quantum free minimizers]\label{improved uniform boundedness}
If $\beta$ satisfies \textup{\textbf{(A1)}$-$\textbf{(A4)}}, then a quantum minimizer $\mathcal{Q}_\hbar$ of the form \eqref{Q equation for property} satisfies the following.
\begin{enumerate}[$(i)$]
\item (Schatten norm bounds)
$$\sup_{\hbar\in(0,1]}\textup{Tr}^\hbar\big(\mathcal{Q}_\hbar^\alpha\big)<\infty\ \textup{for }1\leq\alpha<\infty,\quad\sup_{\hbar\in(0,1]}\|\mathcal{Q}_\hbar\|<\infty.$$
\item (Density and potential function bounds)
$$\sup_{\hbar\in(0,1]}\|U_\hbar\|_{L^{3,\infty}(\mathbb{R}^3)\cap C^{0,\frac{1}{5}}(\mathbb{R}^3)}\lesssim \sup_{\hbar\in(0,1]}\|\rho_{\mathcal{Q}_\hbar}^\hbar\|_{L^1(\mathbb{R}^3)\cap L^{\frac{5}{3}}(\mathbb{R}^3)}<\infty.$$
\item (Lagrange multiplier bound)
$$\inf_{\hbar\in(0,1]}(-\mu_\hbar)>-\infty.$$
\end{enumerate}
\end{lemma}

\begin{proof}
We perform the following iteration. Assume that
\begin{equation}\label{iterpation input}
\sup_{\hbar\in(0,1]}\textup{Tr}^\hbar\big(\mathcal{Q}_\hbar^{\alpha}\big)<\infty
\end{equation}
for some $\alpha\in(\frac{9}{7},3)$. Then, the Hardy-Sobolev and the Lieb-Thirring inequalities imply that
$$\sup_{\hbar\in(0,1]}\|U_\hbar\|_{L^{\frac{3(5\alpha-3)}{3-\alpha}}}\lesssim\sup_{\hbar\in(0,1]}\|\rho_{\mathcal{Q}_\hbar}^\hbar\|_{L^{\frac{5\alpha-3}{3\alpha-1}}}<\infty.$$
Recalling the eigenfunction expansion \eqref{eigenfunction expansion of Q} for $\mathcal{Q}_\hbar$, let $-\mu_j$ be the negative eigenvalues of the Schr\"odinger operator $-\tfrac{\hbar^2}{2}\Delta-U_\hbar$. Then, the dual Lieb-Thirring inequality (Lemma \ref{LT dual}) and \eqref{eigenvalues relation} yield
\begin{equation}\label{temp beta' bound}
\begin{aligned}
\infty&>\sup_{\hbar\in(0,1]}(2\pi\hbar)^3\sum_j \mu_j^{\frac{3(11\alpha-9)}{2(3-\alpha)}}\geq\sup_{\hbar\in(0,1]}(2\pi\hbar)^3\sum_{-\mu_j<-\mu_\hbar} (\mu_\hbar+\beta'(\lambda_j))^{\frac{3(11\alpha-9)}{2(3-\alpha)}}\\
&\geq \sup_{\hbar\in(0,1]}(2\pi\hbar)^3\sum_{-\mu_j<-\mu_\hbar} \beta'(\lambda_j)^{\frac{3(11\alpha-9)}{2(3-\alpha)}}
%=\sup_{\hbar\in(0,1]}(2\pi\hbar)^3\sum_{-\mu_j<-\mu_\hbar} \beta'(\lambda_j)^{\frac{3(11\alpha-9)}{2(3-\alpha)}}
\end{aligned}
\end{equation}
We claim that by Assumption \textup{\textbf{(A4)}},
\begin{equation}\label{beta' lower bound}
\beta'(s)\geq C s^{m-1}
\end{equation}
for all large $s>0$. Indeed, by convexity, the graph of $y=\beta(x)$ is above the graph of its tangent line at $(s,\beta(s))$, i.e., $y=\beta(s)+\beta'(s)(x-s)$. In particular, when $x=0$ and $s>0$ is large, we have $0=\beta(0)\geq \beta(s)+\beta'(s)(-s)$, which implies that $\beta'(s)\geq\beta(s)/s\geq Cs^{m-1}$. Therefore,  applying the claim \eqref{beta' lower bound} to \eqref{temp beta' bound}, we conclude that
\begin{equation}\label{iteration output}
\sup_{\hbar\in(0,1]}\textup{Tr}^\hbar\Big(\mathcal{Q}_\hbar^{\frac{3(11\alpha-9)}{2(3-\alpha)}(m-1)}\Big)=\sup_{\hbar\in(0,1]}(2\pi\hbar)^3\sum_{-\mu_j<-\mu_\hbar} \lambda_j^{\frac{3(11\alpha-9)}{2(3-\alpha)}(m-1)}<\infty.
\end{equation}

By Assumption \textup{\textbf{(A2)}}, we can start the above iteration. Moreover, since $\frac{3(11\alpha-9)}{2(3-\alpha)}(m-1)-\alpha\geq \frac{3(11\alpha-9)}{7(3-\alpha)}-\alpha=\frac{(7\alpha-9)(\alpha+3)}{7(3-\alpha)}\geq \frac{7m-9}{7}$, the exponent increases at least uniformly in each iteration step from \eqref{iterpation input} to \eqref{iteration output}. As a result, the uniform bound \eqref{iterpation input} is derived for all $1\leq\alpha<3$ (here, the case $\alpha=1$ is from the mass constraint).

We denote by $a^-$ a real number which is strictly less than but can be chosen arbitrarily close to $a$. Performing the iteration once again with $\alpha=3^-$, we obtain the bound \eqref{iterpation input} for all $1\leq\alpha<\infty$ (see \eqref{iteration output}), that is, the first part of $(i)$. 
Consequently, the Hardy-Littlewood-Sobolev inequality and the Lieb-Thirring inequality yield
\begin{equation}\label{HLS resulting inequality}
\sup_{\hbar\in(0,1]}\|U_\hbar\|_{L^{3,\infty} \cap L^q} \lesssim \sup_{\hbar\in(0,1]}\|\rho_{\mathcal{Q}_\hbar}^\hbar\|_{L^1\cap L^{\frac{5}{3}^-}}<\infty
\end{equation}
for every $q \in (3, \infty)$. Here we note that a $\hbar$-uniform $L^\infty$ bound for $U_{\hbar}$ does not directly follow from \eqref{HLS resulting inequality},  since due to the Hardy-Littlewood-Sobolev inequality, the implicit constant in \eqref{HLS resulting inequality} may depend on $q$ as $q\to \infty$. Hence, we instead employ the well-known De Giorgi-Nash-Morse estimate to obtain
\begin{equation}\label{GNM est}
|U_\hbar(x)| \leq \|U_\hbar\|_{L^\infty(B_{1/2}(x))} 
\lesssim \|U_\hbar\|_{L^{q_0}(B_1(x))}+ \|\rho^{\hbar}_{\mathcal{Q}_\hbar}\|_{L^{s_0}(B_1(x))} 
\lesssim \|\rho^{\hbar}_{\mathcal{Q}_\hbar}\|_{L^1(\R^3) \cap L^{\frac{5}3^-}(\R^3)}, 
\end{equation}
where $q_0 \in (3,\, \infty)$ and $s_0 \in (\frac{3}{2},\, \frac{5}{3})$ are arbitrary but fixed values. In \eqref{GNM est}, the implicit constant is independent of $x$ and $\hbar$ so that the desired $L^\infty$ bound holds,
\begin{equation}\label{L-infty est}2
\sup_{\hbar \in (0, 1]}\|U_\hbar\|_{L^\infty} 
\lesssim \sup_{\hbar\in(0,1]}\|\rho_{\mathcal{Q}_\hbar}^\hbar\|_{L^1\cap L^{\frac{5}{3}^-}} < \infty.
\end{equation}
We now observe that $-\mu_\hbar-(-\tfrac{\hbar^2}{2}\Delta-U_\hbar)=U_\hbar-(\mu_\hbar-\hbar^2\Delta)\leq \|U_\hbar\|_{L^\infty}$ as operators on $L^2(\mathbb{R}^3)$. Thus, by functional calculus, we see that
$$
\sup_{\hbar\in(0,1]}\|\mathcal{Q}_\hbar\|_{\mathfrak{S}^\infty}\leq \sup_{\hbar\in(0,1]}\Big\{\max_{[0,\|U_\hbar\|_{L^\infty}]}|\tilde{\beta}(s)|\Big\}<\infty,
$$
from which the second part of $(i)$ follows. 

Next, we prove $(ii)$. First, we note that by the Lieb-Thirring inequality with $\alpha=\infty$, the exponent $\frac{5}{3}^-$ in \eqref{L-infty est} can be replaced by $\frac{5}{3}$. It remains to obtain a uniform H\"older semi-norm bound. 
In other words, we need to show
\[
\frac{|U_\hbar(x)-U_\hbar(y)|}{|x-y|^{1/5}} \leq C\|\rho_{\mathcal{Q}_\hbar}^\hbar\|_{L^1\cap L^{\frac{5}{3}}}
\]
for some $C$ independent of $x, y \in \R^3$ and $\hbar$. 
For any choice of $x, y$ such that $|x-y| \geq 1/2$, this immediately follows from the $L^\infty$ estimate of $U_\hbar$. 
Thus, we may assume for fixed $x \in \R^3$, $y \in B_{1/2}(x)$. 
Then, the Calderon-Zigmund estimate and the Sobolev inequality yield
\[
\frac{|U_\hbar(x)-U_\hbar(y)|}{|x-y|^{1/5}} \leq 
\|U_\hbar\|_{C^{0,\frac15}(\overline{B_{1/2}(x)})}\leq C\|U_\hbar\|_{W^{2, 5/3}(B_{1/2}(x))} 
\leq C\|\rho^{\hbar}_{\mathcal{Q}_\hbar}\|_{L^{\frac53}(B_1(x))},
\]
where the constant $C$ is independent of $x$ and $\hbar$. This proves $(ii)$.

For $(iii)$, we claim that $\mu_\hbar\leq \|U_\hbar\|_{L^\infty}$. Indeed, if it is not, $U_\hbar-(\mu_\hbar-\hbar^2\Delta)$ is non-positive. 
Hence, a contradiction $\mathcal{Q}_\hbar=\tilde{\beta}(\mu_\hbar-(-\tfrac{\hbar^2}{2}\Delta-U_\hbar))=0$ is derived.
\end{proof}

As a corollary, we obtain the following profile decomposition.

\begin{corollary}[Profile decomposition for quantum potentials]\label{profile decomp for quantum potential}
Suppose that $\beta$ satisfies \textup{\textbf{(A1)}$-$\textbf{(A4)}}. Let $\mathcal{Q}_\hbar$ be a minimizer for the quantum variational problem \eqref{quantum VP}, and let $\{\hbar_n\}_{n=1}^\infty$ be a sequence such that $\hbar_n\to 0$ as $n\to\infty$. Then, for each $J\in\mathbb{N}$, there exist a subsequence of $\{U_{\hbar_n}\}_{n=1}^\infty$ (but still denoted by $U_{\hbar_n}$), $U^j\in L^{\frac{15}{4}}(\mathbb{R}^3)$, $\{x_n^j\}_{n=1}^\infty$ with $j\in \{1,2,\cdots, J\}$ and $R_n^J\in L^{\frac{15}{4}}(\mathbb{R}^3)$ such that
\begin{equation}\label{profile decomp1}
U_{\hbar_n}=\sum_{j=1}^JU^j(\cdot-x_n^j)+R_n^J,
\end{equation}
\begin{equation}\label{profile decomp2}
\lim_{n\to\infty}|x_n^j-x_n^{j'}|\to \infty\quad\textup{if}\quad j\neq j'
\end{equation}
and
\begin{equation}\label{profile decomp3}
\lim_{J\to\infty}\limsup_{n\to\infty}\|R_n^J\|_{L^r(\mathbb{R}^3)}=0\textup{ for all }\tfrac{15}{4}<r <\infty.
\end{equation}
\end{corollary}

\begin{proof}
By the Hardy-Littlewood-Sobolev inequality and Lemma \ref{improved uniform boundedness} $(ii)$, we have
$$\sup_{\hbar\in(0,1]}\||\nabla|^{-\frac{13}{10}}\rho_{\mathcal{Q}_\hbar}^\hbar\|_{H^1(\mathbb{R}^3)}\sim\sup_{\hbar\in(0,1]}\|\rho_{\mathcal{Q}_\hbar}^\hbar\|_{L^{\frac{15}{4}}(\mathbb{R}^3)\cap L^{\frac{5}{3}}(\mathbb{R}^3)}<\infty.$$
Hence, we apply the profile decomposition \cite[Proposition 3.1]{HmidiKeraani} to $|\nabla|^{\frac{7}{10}}U_{\hbar_n}=(4\pi)|\nabla|^{-\frac{13}{10}}\rho_{\mathcal{Q}_{\hbar_n}}^{\hbar_n}$. Then, passing to a subsequence,
$$|\nabla|^{\frac{7}{10}}U_{\hbar_n}=\sum_{j=1}^JV^j(\cdot-x_n^j)+r_n^J$$
such that \eqref{profile decomp2} holds, $V^j\in H^1$ and $\limsup_{n\to\infty}\|r_n^J\|_{L^r}=0$ as $J\to\infty$ for all $2<r<6$. Thus, inverting $|\nabla|^{\frac{7}{10}}$, we obtain the profile decomposition \eqref{profile decomp1} with $U^j=|\nabla|^{-\frac{7}{10}}V^j$ and $R_n^J=|\nabla|^{-\frac{7}{10}}r_n^J$. Then, \eqref{profile decomp3} follows by the Hardy-Littlewood-Sobolev and Morrey's inequalities.
\end{proof}

Next, we establish the Pohozaev identities for quantum free energy minimizers.

\begin{lemma}[Pohozaev identities]\label{Pohozaev}
Suppose that $\beta$ satisfies \textup{\textbf{(A1)}$-$\textbf{(A4)}}. Let $\mathcal{Q}_\hbar$ be a minimizer for the quantum variational problem \eqref{quantum VP} of the form \eqref{Q equation for property}. Then, 
$$\begin{aligned}
\textup{Tr}^\hbar\big((-\tfrac{\hbar^2}{2}\Delta)\mathcal{Q}_\hbar\big)&=\frac{1}{3}\left\{\mu_\hbar M+\textup{Tr}^\hbar\big((s\beta')(\mathcal{Q}_\hbar)\big)\right\},\\
\int_{\mathbb{R}^6}\frac{\rho_{\mathcal{Q}_\hbar}^\hbar(x)\rho_{\mathcal{Q}_\hbar}^\hbar(y)}{|x-y|}dxdy&=\frac{4}{3}\left\{\mu_\hbar M+\textup{Tr}^\hbar\big((s\beta')(\mathcal{Q}_\hbar)\big)\right\}.
\end{aligned}$$
\end{lemma}

\begin{proof}
By simple algebras, \eqref{eigenvalue equation} and \eqref{eigenvalues relation}, we obtain
\begin{equation}\label{Pohozaev proof 1}
\begin{aligned}
&\textup{Tr}^\hbar\big((-\tfrac{\hbar^2}{2}\Delta)\mathcal{Q}_\hbar\big)-\textup{Tr}^\hbar(U_\hbar\mathcal{Q}_\hbar)\\
&=\textup{Tr}^\hbar\big((-\tfrac{\hbar^2}{2}\Delta-U_\hbar)\mathcal{Q}_\hbar\big)=\textup{Tr}^\hbar\left(\sum\lambda_j\ket{(-\tfrac{\hbar^2}{2}\Delta-U_\hbar)\phi_j}\bra{\phi_j}\right)\\
&=\textup{Tr}^\hbar\left(\sum\lambda_j\ket{-\mu_j\phi_j}\bra{\phi_j}\right)=-(2\pi\hbar)^3\sum\lambda_j\mu_j.
\end{aligned}
\end{equation}
On the other hand, we have
$$\langle (-\tfrac{\hbar^2}{2}\Delta-U_\hbar+\mu_j)\phi_j, x\cdot\nabla\phi_j\rangle=0.$$
By integration by parts, we calculate 
$$\begin{aligned}
\langle \Delta\phi_j|x\cdot\nabla\phi_j\rangle&=\sum_{k,\ell=1}^3\int_{\mathbb{R}^3} \partial_{x_k}^2\phi_j (x_\ell\partial_{x_\ell}\phi_j)dx\\
&=-\sum_{k,\ell=1}^3\int_{\mathbb{R}^3}\delta_{k\ell}\partial_{x_k}\phi_j \partial_{x_\ell}\phi_jdx-\sum_{k,\ell=1}^3\int_{\mathbb{R}^3} x_\ell \partial_{x_k}\phi_j \partial_{x_\ell}\partial_{x_k}\phi_jdx\\
&=- \int_{\mathbb{R}^3} |\nabla\phi_j|^2 dx-\frac{1}{2}\int_{\mathbb{R}^3} x\cdot\nabla(|\nabla\phi_j|^2) dx=\frac{1}{2} \int_{\mathbb{R}^3} |\nabla\phi_j|^2 dx=\frac{1}{2}\|\nabla\phi_j\|_{L^2}^2.
\end{aligned}$$
Similarly,
$$\langle U_\hbar\phi_j|x\cdot\nabla\phi_j\rangle=\frac{1}{2}\int_{\mathbb{R}^3} U_\hbar \left(x\cdot\nabla(\phi_j^2)\right)dx=-\frac{1}{2}\int_{\mathbb{R}^3} (x\cdot\nabla U_\hbar) \phi_j^2dx-\frac{3}{2}\int_{\mathbb{R}^3} U_\hbar (\phi_j^2)dx,$$
and
$$\langle \phi_j|x\cdot\nabla\phi_j\rangle=\frac{1}{2}\int  _{\mathbb{R}^3}x\cdot\nabla(\phi_j^2) dx=-\frac{3}{2}\int_{\mathbb{R}^3} \phi_j^2 dx=-\frac{3}{2}.$$
Thus, inserting these expressions to the above identity and summing them up, we obtain
$$\begin{aligned}
0&=(2\pi\hbar)^3\sum \lambda_j\langle (-\tfrac{\hbar^2}{2}\Delta-U_\hbar+\mu_j)\phi_j| x\cdot\nabla\phi_j\rangle\\
&=(2\pi\hbar)^3\sum\lambda_j\left\{-\frac{\hbar^2}{4} \|\nabla\phi_j\|_{L^2}^2+\frac{1}{2}\int_{\mathbb{R}^3} (x\cdot\nabla U_\hbar) \phi_j^2dx+\frac{3}{2}\int_{\mathbb{R}^3} U_\hbar (\phi_j^2)dx-\frac{3}{2}\mu_j\right\}\\
&=-\frac{1}{2}\textup{Tr}^\hbar\big((-\tfrac{\hbar^2}{2}\Delta)\mathcal{Q}_\hbar\big)+\frac{1}{2}\int_{\mathbb{R}^3} (x\cdot\nabla U_\hbar) \rho_{\mathcal{Q}_\hbar}^\hbar dx+\frac{3}{2}\int_{\mathbb{R}^3} U_\hbar\rho_{\mathcal{Q}_\hbar}^\hbar dx-\frac{3}{2}(2\pi\hbar)^3\sum \lambda_j\mu_j.
\end{aligned}$$
The second term is simplified by symmetrization, 
$$\begin{aligned}
&\int_{\mathbb{R}^3} (x\cdot\nabla U_\hbar) \rho_{\mathcal{Q}_\hbar}^\hbar dx\\
&=-\int_{\mathbb{R}^3} \left\{x\cdot \int_{\mathbb{R}^3}\frac{x-y}{|x-y|^3}\rho_{\mathcal{Q}_\hbar}^\hbar(y)dy\right\} \rho_{\mathcal{Q}_\hbar}^\hbar(x) dx=-\int_{\mathbb{R}^6} \frac{x\cdot(x-y)}{|x-y|^3} \rho_{\mathcal{Q}_\hbar}^\hbar(x)  \rho_{\mathcal{Q}_\hbar}^\hbar(y)dxdy\\
&=-\frac{1}{2}\int_{\mathbb{R}^6} \frac{x\cdot(x-y)}{|x-y|^3} \rho_{\mathcal{Q}_\hbar}^\hbar(x)  \rho_{\mathcal{Q}_\hbar}^\hbar(y)dxdy -\frac{1}{2}\int_{\mathbb{R}^6} \frac{y\cdot(y-x)}{|x-y|^3} \rho_{\mathcal{Q}_\hbar}^\hbar(x)  \rho_{\mathcal{Q}_\hbar}^\hbar(y)dxdy\\
&=-\frac{1}{2}\int_{\mathbb{R}^6}\frac{\rho_{\mathcal{Q}_\hbar}^\hbar(x)  \rho_{\mathcal{Q}_\hbar}^\hbar(y)}{|x-y|}dxdy=-\frac{1}{2}\int_{\mathbb{R}^3} U_\hbar\rho_{\mathcal{Q}_\hbar}^\hbar dx.
\end{aligned}$$
Thus, it follows that 
\begin{equation}\label{Pohozaev proof 2}
0=-\frac{1}{2}\textup{Tr}^\hbar\big((-\tfrac{\hbar^2}{2}\Delta)\mathcal{Q}_\hbar\big)+\frac{5}{4}\textup{Tr}^\hbar(U_\hbar\mathcal{Q}_\hbar)-\frac{3}{2}(2\pi\hbar)^3\sum \lambda_j\mu_j.
\end{equation}
Note that by \eqref{eigenvalues relation},
$$\begin{aligned}
(2\pi\hbar)^3\sum \lambda_j\mu_j&=(2\pi\hbar)^3\sum \lambda_j(\mu_\hbar+\beta'(\lambda_j))=\mu_\hbar(2\pi\hbar)^3\sum \lambda_j+(2\pi\hbar)^3\sum \lambda_j\beta'(\lambda_j)\\
&=\mu_\hbar \textup{Tr}^\hbar(\mathcal{Q}_\hbar)+\textup{Tr}^\hbar\big((s\beta')(\mathcal{Q}_\hbar)\big)=\mu_\hbar M+\textup{Tr}^\hbar\big((s\beta')(\mathcal{Q}_\hbar)\big).
\end{aligned}$$
Thus, solving \eqref{Pohozaev proof 1} and \eqref{Pohozaev proof 2} for $\textup{Tr}^\hbar\big((-\tfrac{\hbar^2}{2}\Delta)\mathcal{Q}_\hbar\big)$ and $\textup{Tr}^\hbar(U_\hbar\mathcal{Q}_\hbar)$, we prove the lemma.
\end{proof}

As an application of the above Pohozaev identities, we obtain a uniform upper bound for the Lagrange multiplier $-\mu_\hbar$ in \eqref{Q equation for property}.
\begin{corollary}[Lower bound for Lagrange multipliers]\label{Lagrange multiplier bound}
Suppose that $\beta$ satisfies \textup{\textbf{(A1)}$-$\textbf{(A4)}}. Let $\mathcal{Q}_\hbar$ be a minimizer for the variational problem \eqref{quantum VP} of the form \eqref{Q equation for property}. Then, 
$$\sup_{\hbar\in(0,1]}(-\mu_\hbar)<0. $$
\end{corollary}

\begin{proof}
It is known that $-\mu_\hbar<0$ for each $\hbar>0$. Thus, it suffices to show that
$$\limsup_{\hbar\to0}(-\mu_\hbar)<0.$$
Indeed, by the Pohozaev identities and Assumption \textup{\textbf{(A4)}}, that is $\beta(s)\geq\frac{1}{3}s\beta'(s)$, we have 
$$\begin{aligned}
\mathcal{J}_{\min}^\hbar&=\textup{Tr}^\hbar\big((-\tfrac{\hbar^2}{2}\Delta)\mathcal{Q}_\hbar\big)-\frac{1}{2}\int_{\mathbb{R}^6}\frac{\rho_{\mathcal{Q}_\hbar}^\hbar(x)\rho_{\mathcal{Q}_\hbar}^\hbar(y)}{|x-y|}dxdy+\textup{Tr}^\hbar\big(\beta(\mathcal{Q})\big)\\
&=-\frac{1}{3}\mu_\hbar M+\textup{Tr}^\hbar\left((\beta-\tfrac{1}{3}(s\beta'))(\mathcal{Q}_\hbar)\right)\\
&\geq -\frac{1}{3}\mu_\hbar M.
\end{aligned}$$
Thus, since $\underset{\hbar\to0}\limsup\mathcal{J}_{\min}^\hbar\leq\mathcal{J}_{\min}<0$ (see \eqref{min free energy 1}), we prove the desired bound.
\end{proof}

\section{Classical distributions from quantum free energy minimizers}

In this section, we establish the opposite inequality to \eqref{min free energy 1}. Precisely, reversing the roles of classical and quantum states in Section \ref{sec: CtoQ}, we employ classical distribution functions made of a quantum free energy minimizer, and show the following proposition.

\begin{proposition}[Classical states from a quantum minimizer]\label{prop: QtoC}
Suppose that $\beta$ satisfies  \textup{\textbf{(A1)}$-$\textbf{(A4)}}. Given a minimizer 
$$\mathcal{Q}_\hbar=\tilde{\beta}\big(-\mu_{\hbar}-(-\tfrac{\hbar^2}{2}\Delta-U_\hbar)\big)$$
for the quantum minimization problem \eqref{quantum VP} with small $\hbar>0$, we construct a distribution function
\begin{equation}\label{CS from quantum minimizer}
f_\hbar(q,p):=\tilde{\beta}\big(-\mu_{\hbar}-(\tfrac{|p|^2}{2}-U_\hbar(q))\big).
\end{equation}
Then, for any sequence $\{\hbar_n\}_{n=1}^\infty$ with $\hbar_n\to0$, there exists a subsequence but still denoted by 
$\{\hbar_n\}_{n=1}^\infty$ such that 
$$\lim_{n\to\infty}\left\{\mathcal{J}(f_{\hbar_n})-\mathcal{J}_{\min}^{\hbar_n}\right\}=0\quad\textup{and}\quad \lim_{n\to\infty}\mathcal{M}(f_{\hbar_n})=M.$$
As a consequence, we have
\begin{equation}\label{min free energy 2}
\liminf_{n\to\infty}\mathcal{J}_{\min}^{\hbar_n}\geq\mathcal{J}_{\min}.
\end{equation}
\end{proposition}

Before giving a proof, we would like to point out a technical difficulty in proving the proposition. Now, the potential function $-U_\hbar$ depends on $\hbar$. It may vary in the semi-classical limit $\hbar\to0$. Thus, if one attempts to repeat the same argument in Section \ref{sec: CtoQ}, one would immediately encounter a problem in application of Weyl's law (Theorem \ref{Weyl's law}).

In order to overcome this subtle issue, we employ the profile decomposition. Indeed, given a sequence $\{\hbar_n\}_{n=1}^\infty$ with $\hbar_n\to 0$, passing to a subsequence, a sequence of potential functions $\{-U_{\hbar_n}\}_{n=1}^\infty$ has a profile decomposition (Corollary \ref{profile decomp for quantum potential}). The following theorem asserts that Weyl's law still holds for such a sequence of potential functions.

\begin{proposition}[Weyl's law for profile decompositions]\label{Weyl's law for profile}
Let $\mu>0$ and $r>\frac{3}{2}$. Suppose that a sequence of non-negative potential functions $\{U_{\hbar_n}\}_{n=1}^\infty\subset L^r(\mathbb{R}^3)$ can be written as 
\begin{equation}\label{Weyl's law profile decomposition}
U_{\hbar_n}=\sum_{j=1}^J U^j(\cdot-x_n^j)+R_\hbar^J
\end{equation}
such that $U^j\in L^r(\mathbb{R}^3)$, $|x_n^j-x_n^{j'}|\to\infty$ for $j\neq j'$, and
$$\lim_{J\to\infty}\limsup_{n\to\infty}\|R_n^J\|_{L^r}=0.$$
Then, for all $E\geq\mu$, we have
\begin{equation}\label{eq: Weyl's law for profiles}
\begin{aligned}
\lim_{n\to\infty}(2\pi\hbar_n)^3N(-E,\hbar_n, -U_{\hbar_n})&-\Big|\Big\{(q,p)\in\mathbb{R}^6:\ \tfrac{|p|^2}{2}-U_{\hbar_n}(q)\leq-E\Big\}\Big|=0.
\end{aligned}
\end{equation}
\end{proposition}

\begin{lemma}\label{classical CLR perturbation}
Suppose that $V\in L^r(\mathbb{R}^3)$ for some $r\geq\frac{3}{2}$. Then, we have
$$\Big|\Big\{(q,p)\in\mathbb{R}^6: \left|\tfrac{1}{2}|p|^2+V(q)+E\right|\leq\delta\Big\}\Big|\lesssim\delta E^{-(r-\frac{1}{2})}\|V_-\|_{L^r}^r$$
for $E>0$ and $\delta\in(0,\frac{E}{2}]$.
\end{lemma}

\begin{proof}
Let $\tilde{V}(x):=\max\{-V(x), \frac{E}{2}\}$. Then, $|\tfrac{1}{2}|p|^2+V(q)+E|\leq \delta$ if and only if $||p|^2-\tilde{V}(q)+E\big|\leq \delta$, because $0<\delta\leq\frac{E}{2}$. Hence, we have
$$\begin{aligned}
&\Big\{(q,p): \left|\tfrac{|p|^2}{2}+V(q)+E\right|\leq \delta\Big\}=\Big\{(q,p): \big|\tfrac{|p|^2}{2}-\tilde{V}(q)+E\big|\leq \delta\Big\}\\
&=\Big\{(q,p): \tfrac{|p|^2}{2}-\tilde{V}(q)\leq-E+\delta\Big\}-\Big\{(q,p): \tfrac{|p|^2}{2}-\tilde{V}(q)\leq-E-\delta\Big\}.
\end{aligned}$$
It implies that 
$$\begin{aligned}
&\Big|\Big\{(q,p): \left|\tfrac{|p|^2}{2}+V(q)+E\right|\leq \delta\Big\}\Big|\\
&=\frac{8\sqrt{2}\pi}{3}\int_{\tilde{V}(x)\geq E-\delta} (\tilde{V}(x)-E+\delta)^{\frac{3}{2}}dx-\frac{8\sqrt{2}\pi}{3}\int_{\tilde{V}(x)\geq E+\delta} (\tilde{V}(x)-E-\delta)^{\frac{3}{2}}dx\\
&=\frac{8\sqrt{2}\pi}{3}\int_{E-\delta\leq \tilde{V}(x)\leq E+\delta} (\tilde{V}(x)-E+\delta)^{\frac{3}{2}}dx\\
&\quad+\frac{8\sqrt{2}\pi}{3}\int_{\tilde{V}(x)\geq E+\delta} \left\{(\tilde{V}(x)-E+\delta)^{\frac{3}{2}}-(\tilde{V}(x)-E-\delta)^{\frac{3}{2}}\right\}dx.
\end{aligned}$$
In the first integral,
$$0\leq (\tilde{V}(x)-E+\delta)^{\frac{3}{2}}\leq 2\delta(\tilde{V}(x)-E+\delta)^{\frac{1}{2}}\leq 2\delta\tilde{V}(x)^{\frac{1}{2}} \leq 2\delta (\tfrac{E}{2})^{-(r-\frac{1}{2})}\tilde{V}(x)^r,$$
while in the second integral, by the fundamental theorem of calculus, 
$$\begin{aligned}
&(\tilde{V}(x)-E+\delta)^{\frac{3}{2}}-(\tilde{V}(x)-E-\delta)^{\frac{3}{2}}\\
&=\int_{-1}^1\frac{d}{dt}\left\{(\tilde{V}(x)-E+t\delta)^{\frac{3}{2}} \right\}dt=\int_{-1}^1\frac{3\delta}{2}(\tilde{V}(x)-E+t\delta)^{\frac{1}{2}}dt\\
&\leq 3\delta (\tilde{V}(x)-E+\delta)^{\frac{1}{2}}\leq 3\delta (\tfrac{E}{2})^{-(r-\frac{1}{2})}\tilde{V}(x)^r.
\end{aligned}$$
Thus, we prove that 
$$\Big|\Big\{(q,p): \left|\tfrac{1}{2}|p|^2+V(q)+E\right|\leq \delta\Big\}\Big|\lesssim \delta E^{-(r-\frac{1}{2})}\int_{\tilde{V}(x)\geq E-\delta} \tilde{V}(x)^r dx\leq \delta E^{-(r-\frac{1}{2})}\|V\|_{L^r}^r.$$
\end{proof}

\begin{proof}[Proof of Proposition \ref{Weyl's law for profile}]
Given arbitrarily small $\epsilon>0$, we take $J=J(\epsilon)\geq1$ and $\{\tilde{U}^j\}_{j=1}^J\subset C_c^\infty$ such that
$\limsup_{n\to\infty}\|R_n^J\|_{L^r}\leq\epsilon$ and $\|\tilde{U}^j-U^j\|_{L^r}\leq\frac{\epsilon}{J}$, and we modify the profile decomposition \eqref{Weyl's law profile decomposition} as
\begin{equation}\label{Weyl's law profile decomposition'}
U_{\hbar_n}=\tilde{U}_n+\tilde{R}_n^J,
\end{equation}
where
$\tilde{U}_n:=\sum_{j=1}^J\tilde{U}^j(\cdot-x_n^j)$ and $\tilde{R}_n^J:=R_n^J+\sum_{j=1}^J(U^j-\tilde{U}^j)(\cdot-x_n^j)$
and
$$\limsup_{n\to\infty}\|\tilde{R}_n^J\|_{L^r}\leq2\epsilon.$$
Let $\delta\in(0,\frac{\mu}{2}]$. Repeating \eqref{L^r Weyl's law - one direction}, we write 

\begin{equation}\label{eq: Weyl's law profile proof}
\begin{aligned}
N(-E,\hbar_n,-U_{\hbar_n})&\leq\textup{dim}\Big\{\textup{Ran}\ \mathbf{1}_{(-\infty,0)}(-\tfrac{(1-\delta)\hbar_n^2}{2}\Delta-\tilde{U}_n+E-\delta)\Big\}\\
&\quad+\textup{dim}\Big\{\textup{Ran}\ \mathbf{1}_{(-\infty,0)}(-\tfrac{\delta\hbar_n^2}{2}\Delta+\delta-\tilde{R}_n^J)\Big\} \\
&=\textup{dim}\Big\{\textup{Ran}\ \mathbf{1}_{(-\infty,-\frac{E-\delta}{1-\delta})}(-\tfrac{\hbar_n^2}{2}\Delta-\tfrac{1}{1-\delta}\tilde{U}_n)\Big\}\\
&\quad+\textup{dim}\Big\{\textup{Ran}\ \mathbf{1}_{(-\infty,-1)}(-\tfrac{\hbar_n^2}{2}\Delta-\tfrac{1}{\delta}\tilde{R}_n^J)\Big\} \\
&= N\Big(-\tfrac{E-\delta}{1-\delta},\hbar_n,-\tfrac{1}{1-\delta}\tilde{U}_n\Big)+N\Big(-1, \hbar_n, -\tfrac{1}{\delta}\tilde{R}_n^J\Big).
\end{aligned}
\end{equation}
For the second term one the right hand side of \eqref{eq: Weyl's law profile proof}, we apply Lemma \ref{CLR bound},
$$\limsup_{n\to\infty}(2\pi\hbar_n)^3N\left(-1, \hbar_n, -\tfrac{1}{\delta}\tilde{R}_n^J\right)=O\left(\frac{1}{\delta^r}\right)\limsup_{n\to\infty}\|\tilde{R}_n^J\|_{L^r}^r=O\left(\frac{\epsilon^r}{\delta^r}\right).$$
For the first term, Theorem \ref{Weyl's law} yields
$$\begin{aligned}
&(1-\delta)^{\frac{3}{2}}\limsup_{n\to\infty}(2\pi\hbar_n)^3N\left(-\tfrac{E-\delta}{1-\delta},\hbar_n, -\tfrac{1}{1-\delta}\tilde{U}_n\right)\\
&=(1-\delta)^{\frac{3}{2}}\Big|\Big\{(q,p): \tfrac{|p|^2}{2}-\tfrac{1}{1-\delta}\tilde{U}_n(q)\leq-\tfrac{E-\delta}{1-\delta}\Big\}\Big|\\
&=\Big|\Big\{(q,p): \tfrac{|p|^2}{2}-\tilde{U}_n(q)\leq-E+\delta\Big\}\Big|\\
&=\Big|\Big\{(q,p): \tfrac{|p|^2}{2}-\tilde{U}_n(q)\leq-E-\delta\Big\}\Big|+\Big|\Big\{(q,p): |\tfrac{|p|^2}{2}-\tilde{U}_n(q)+E|\leq \delta\Big\}\Big|
\end{aligned}$$
for sufficiently large $n$. Here, when Theorem \ref{Weyl's law} is applied, we used that the measure of the set $\{(q,p): \frac{|p|^2}{2}-\tfrac{1}{1-\delta}\tilde{U}_n(q)\leq-\tfrac{E-\delta}{1-\delta}\}$ does not depend on large $n$, because the supports of $\tilde{U}^j(\cdot-x_n^j)$ are mutually disjoint. Thus, it follows from Lemma \ref{classical CLR perturbation} that 
$$\begin{aligned}
&(1-\delta)^{\frac{3}{2}}\limsup_{n\to\infty}(2\pi\hbar_n)^3N\left(-\tfrac{E-\delta}{1-\delta},\hbar_n, -\tfrac{1}{1-\delta}\tilde{U}_n\right)\\
&=\left|\left\{(q,p): \tfrac{|p|^2}{2}-\tilde{U}_n(q)\leq-(E+\delta)\right\}\right|+O\left(\delta\right)\|\tilde{U}_n\|_{L^r}^r.
\end{aligned}$$
Therefore, coming back to \eqref{eq: Weyl's law profile proof}, we obtain 
\begin{equation}\label{eq: Weyl's law profile proof2}\begin{aligned}
&\limsup_{n\to\infty}(2\pi\hbar_n)^3N(-E,\hbar_n, -U_\hbar)\\
&\leq \frac{1}{(1-\delta)^{\frac{3}{2}}}\left|\left\{(q,p): \tfrac{|p|^2}{2}-\tilde{U}_n(q)\leq-(E+\delta)\right\}\right|+O(\delta)+O\left(\frac{\epsilon^r}{\delta^r}\right).
\end{aligned}
\end{equation}
Next, we aim to replace $\tilde{U}_n$ by $U_{\hbar_n}$ in the bound \eqref{eq: Weyl's law profile proof2}. To this end,   we write
$$\begin{aligned}
&\Big|\Big\{(q,p):\ \tfrac{|p|^2}{2}-\tilde{U}_n(q)\leq-(E+\delta)\Big\}\Big|\\
&=(1+\delta)^{\frac{3}{2}}\Big|\Big\{(q,p):\ \tfrac{1+\delta}{2}|p|^2-\tilde{U}_n(q)\leq-(E+\delta)\Big\}\Big|\\
&\leq(1+\delta)^{\frac{3}{2}}\Big|\Big\{(q,p):\ \tfrac{|p|^2}{2}-U_{\hbar_n}(q)\leq -E\Big\}\Big|+(1+\delta)^{\frac{3}{2}}\Big|\Big\{(q,p):\ \tfrac{\delta}{2}|p|^2-\tilde{R}_n^J(q)\leq-\delta\Big\}\Big|\\
&=(1+\delta)^{\frac{3}{2}}\Big|\Big\{(q,p):\ \tfrac{|p|^2}{2}-U_{\hbar_n}(q)\leq -E\Big\}\Big|+(1+\delta)^{\frac{3}{2}}\Big|\Big\{(q,p):\ 1+\tfrac{|p|^2}{2}\leq\tfrac{1}{\delta}\tilde{R}_n^J(q)\Big\}\Big|.
\end{aligned}$$
Then, applying Lemma \ref{classical CLR} to the second term, we obtain
$$\begin{aligned}
&\Big|\Big\{(q,p):\ \tfrac{|p|^2}{2}-\tilde{U}_n(q)\leq-(E+\delta)\Big\}\Big|\\
&\lesssim (1+\delta)^{\frac{3}{2}}\Big|\Big\{(q,p):\ \tfrac{|p|^2}{2}-U_{\hbar_n}(q)\leq -E\Big\}\Big|+O\left(\frac{\epsilon^r}{\delta^r}\right).
\end{aligned}$$
Combining this inequality with \eqref{eq: Weyl's law profile proof2}, we obtain
$$\begin{aligned}
&\limsup_{n\to\infty}(2\pi\hbar_n)^3N(-E,\hbar_n, -U_{\hbar_n})\\
&\leq \left(\frac{1+\delta}{1-\delta}\right)^{\frac{3}{2}}\left|\left\{(q,p):\ \tfrac{|p|^2}{2}-U_{\hbar_n}(q)\leq -E\right\}\right|+O\left(\frac{\epsilon^r}{\delta^r}\right)+O(\delta)
\end{aligned}$$
for all sufficiently large $n$. As a consequence, we have
$$\begin{aligned}
&\limsup_{n\to\infty}(2\pi\hbar_n)^3N(-E,\hbar_n, -U_{\hbar_n})\\
&\leq \left(\frac{1+\delta}{1-\delta}\right)^{\frac{3}{2}}\liminf_{n\to\infty}\left|\left\{(q,p):\ \tfrac{|p|^2}{2}-U_{\hbar_n}(q)\leq -E\right\}\right|+O\left(\frac{\epsilon^r}{\delta^r}\right)+O(\delta).
\end{aligned}$$
Note that the bound is independent of large $n$. Therefore, sending $\epsilon\to 0$ and then $\delta\to 0$, we prove that 
$$\limsup_{n\to\infty}(2\pi\hbar_n)^3N(-E,\hbar_n, -U_{\hbar_n})\leq \liminf_{n\to\infty}\left|\left\{(q,p):\ \tfrac{|p|^2}{2}-U_{\hbar_n}(q)\leq -E\right\}\right|.$$
The opposite inequality, that is,
$$\liminf_{n\to\infty}(2\pi\hbar_n)^3N(-E,\hbar_n, -U_{\hbar_n})\geq\limsup_{n\to\infty}\left|\left\{(q,p):\ \tfrac{|p|^2}{2}-U_{\hbar_n}(q)\leq -E\right\}\right|,$$
can be proved similarly. 
\end{proof}

Now, we are ready to prove Proposition \ref{prop: QtoC}.
\begin{proof}[Proof of Proposition \ref{prop: QtoC}]
The proof closely follows from that of Proposition \ref{prop: CtoQ}. Indeed, it suffices to show convergences analogous to those in Lemma \ref{CtoQ Casimir} and \ref{CtoQ potential energy} such that $\gamma_\hbar=\tilde{\beta}(-\mu-(-\tfrac{\hbar^2}{2}\Delta-U))$ and $\mathcal{Q}=\tilde{\beta}(-\mu-(\tfrac{|p|^2}{2}-U(q)))$ are replaced by $\mathcal{Q}_\hbar=\tilde{\beta}(-\mu_{\hbar}-(-\tfrac{\hbar^2}{2}\Delta-U_\hbar))$ and $f_\hbar=\tilde{\beta}(-\mu_{\hbar}-(\tfrac{|p|^2}{2}-U_\hbar(q)))$, respectively.

First we note that, passing to a subsequence, we may assume that $\mu_{\hbar_n}\to \tilde\mu<0$ as $n\to \infty$. Next, repeating the proof of Lemma \ref{CtoQ Casimir}, we prove that 
$$\lim_{n\to\infty}\mathcal{M}(f_{\hbar_n})=M\quad\textup{and}\quad\lim_{n\to\infty}\mathcal{C}(f_{\hbar_n})-\mathcal{C}_{\hbar_n}(\mathcal{Q}_{\hbar_n})=0$$
and
$$\lim_{n\to \infty}\iint\big(\tfrac{|p|^2}{2}-U_{\hbar_n}(q)\big)f_{\hbar_n}(q,p)dqdp-\textup{Tr}^{\hbar_n}\big((-\tfrac{\hbar_n^2}{2}\Delta-U_{\hbar_n})\mathcal{Q}_{\hbar_n}\big)=0.$$
Note that here, Proposition \ref{Weyl's law for profile} should be employed instead of Theorem \ref{Weyl's law}, because the potential function $-U_{\hbar_n}$ is also depends on the parameter $\hbar_n$. As a consequence, the interpolation estimate (Theorem \ref{interpolation inequality}) deduces the uniform bound
$$\sup_{\hbar\in(0,1]}\|\rho_{f_\hbar}\|_{L^1\cap L^{\frac{5}{3}}}<\infty.$$

It remains to show that
\begin{align}
\lim_{\hbar\to0}\int_{\mathbb{R}^6}\frac{\rho_{f_\hbar}(x)\rho_{f_\hbar}(y)}{|x-y|}dxdy-\int_{\mathbb{R}^6}\frac{\rho_{\mathcal{Q}_\hbar}^\hbar(x)\rho_{\mathcal{Q}_\hbar}^\hbar(y)}{|x-y|}dxdy&=0,\label{eq: CtoQ potential energy convergence}\\
\lim_{\hbar\to0}\int_{\mathbb{R}^6}\frac{\rho_{f_\hbar}(x)\rho_{\mathcal{Q}_\hbar}^\hbar(y)}{|x-y|}dxdy-\int_{\mathbb{R}^6}\frac{\rho_{\mathcal{Q}_\hbar}^\hbar(x)\rho_{\mathcal{Q}_\hbar}^\hbar(y)}{|x-y|}dxdy&=0,\nonumber
\end{align}
which are similar to Lemma \ref{CtoQ potential energy}. We again show \eqref{eq: CtoQ potential energy convergence} only. Repeating, we decompose 
\begin{equation}\label{potential energy decomposition'}\begin{aligned}
&\int_{\mathbb{R}^6}\frac{\rho_{\mathcal{Q}_\hbar}^\hbar(x)\rho_{\mathcal{Q}_\hbar}^\hbar(y)}{|x-y|}dxdy-\int_{\mathbb{R}^6}\frac{\rho_{f_\hbar}(x)\rho_{f_\hbar}(y)}{|x-y|}dxdy\\
&=\int_{\mathbb{R}^6}\frac{(\rho_{\mathcal{Q}_\hbar}^\hbar-\rho^\hbar_{\textup{Op}_\hbar^T[f_\hbar]})(x)(\rho_{\mathcal{Q}_\hbar}^\hbar+\rho_{f_\hbar})(y)}{|x-y|}dxdy\\
&\quad+\int_{\mathbb{R}^6}\frac{(\rho^\hbar_{\textup{Op}_\hbar^T[f_\hbar]}-\rho_{f_\hbar})(x)(\rho_{\mathcal{Q}_\hbar}^\hbar+\rho_{f_\hbar})(y)}{|x-y|}dxdy.
\end{aligned}
\end{equation}
We can prove that the first integral in \eqref{potential energy decomposition'} converges to zero as we handled $I$ in \eqref{potential energy decomposition}. However, unlike $II$ in \eqref{potential energy decomposition}, the convergence $\rho^\hbar_{\textup{Op}_\hbar^T[f_\hbar]}-\rho_{f_\hbar}\to 0$ in $L^{\frac{6}{5}}$ does not immediately follow from Proposition \ref{TQ-conserv} due to the $\hbar$-dependence of $f_\hbar$. 
Thus, we rather employ the Plancherel theorem and H\"older and Hausdorff-Young inequalities,  
$$\begin{aligned}
\left|\int_{\mathbb{R}^6}\frac{(\rho^\hbar_{\textup{Op}_\hbar^T[f_\hbar]}-\rho_{f_\hbar})(x)(\rho_{\mathcal{Q}_\hbar}^\hbar+\rho_{f_\hbar})(y)}{|x-y|}dxdy\right|&\sim \left|\int_{\mathbb{R}^3} \frac{e^{-\hbar|\xi|^2}-1}{|\xi|^2}\overline{\widehat{\rho_{f_\hbar}}(\xi)}(\widehat{\rho_{\mathcal{Q}_\hbar}^\hbar}+\widehat{\rho_{f_\hbar}})(\xi) d\xi\right|\\
&\leq \left\|\tfrac{e^{-\hbar|\xi|^2}-1}{|\xi|^2}\right\|_{L^5}\|\widehat{\rho_{f_\hbar}}\|_{L^{\frac{5}{2}}}\|\widehat{\rho_{{\mathcal{Q}_\hbar}}^\hbar}+\widehat{\rho_{f_\hbar}}\|_{L^{\frac{5}{2}}}\\
&\leq \left\|\tfrac{e^{-\hbar|\xi|^2}-1}{|\xi|^2}\right\|_{L^5}\|\rho_{f_\hbar}\|_{L^{\frac{5}{3}}}\|\rho_{{\mathcal{Q}_\hbar}}^\hbar+\rho_{f_\hbar}\|_{L^{\frac{5}{3}}}.
\end{aligned}$$
Finally, using that 
$$\begin{aligned}
\left\|\tfrac{e^{-\hbar|\xi|^2}-1}{|\xi|^2}\right\|_{L^5}&\leq \left\|\tfrac{e^{-\hbar|\xi|^2}-1}{|\xi|^2}\right\|_{L^5(|\xi|\leq\hbar^{-1/2})}+\left\|\tfrac{e^{-\hbar|\xi|^2}-1}{|\xi|^2}\right\|_{L^5(|\xi|\geq\hbar^{-1/2})}\\
&\sim \|\hbar\|_{L^5(|\xi|\leq\hbar^{-1/2})}+\left\|\tfrac{1}{|\xi|^2}\right\|_{L^5(|\xi|\geq\hbar^{-1/2})}\sim\hbar^{\frac{7}{10}},
\end{aligned}$$
we conclude that the second integral in \eqref{potential energy decomposition'} converges to zero.
\end{proof}

\section{Proof of the main theorem}
In this section, we complete the proof of Theorem \ref{thm1} $(ii)$-$(iv)$. Theorem \ref{thm1} $(i)$ is proved at the end of Section \ref{sec: CtoQ}.

%\begin{theorem}\label{potential correspondence}
%Let 
%$$\left\{\mathcal{Q}_\hbar= (\beta')^{-1}\big(-\mu_\hbar-(-\tfrac{\hbar^2}{2}\Delta-U_\hbar)\big)_+: 0<\hbar\leq 1\right\}$$
%be a family of minimizers for the quantum variational problem \eqref{quantum VP}. Then, there exist sequences $\{\hbar_n\}_{n=1}^\infty\subset(0,1]$, $\{x_n\}_{n=1}^\infty \subset \mathbb{R}^{3}$ and a minimizer
%$$\mathcal{Q} = (\beta')^{-1}\big(-\mu-(\tfrac{|p|^2}{2}-U(q))\big)_+$$ 
%for the classical variational problem \eqref{classical VP} such that
%$\hbar_n\to 0$, $\mu_{\hbar_n} \to \mu$ and
%$$\|\nabla U_{\hbar_n}(\cdot-x_n)-\nabla U\|_{L^2}\to 0$$
%as $n\to\infty$.
%\end{theorem}

\begin{proof}[Proof of Theorem \ref{thm1} $(ii)$]
Let $\{\mathcal{Q}_\hbar\}_{\hbar\in(0,1]}$ be a family of minimizers for the quantum minimization problem \eqref{quantum VP}. Then, by Proposition \ref{prop: CtoQ} and \ref{prop: QtoC}, we have
$$\lim_{\hbar\to0}\mathcal{J}_{\min}^\hbar=\mathcal{J}_{\min}.$$
\end{proof}

\begin{proof}[Proof of Theorem \ref{thm1} $(iii)$]
By Theorem \ref{thm1} $(i)$, there exists a sequence $\{\hbar_n\}_{n=1}^\infty$ such that
$$\lim_{n\to\infty}\mathcal{J}(f_{\hbar_n})=\mathcal{J}_{\min},$$
where
$$f_{\hbar}=\tilde{\beta}\big(-\mu_{\hbar}-(\tfrac{|p|^2}{2}-U_{\hbar}(q))\big)$$
and $U_\hbar=\frac{1}{|x|}*\rho_{\mathcal{Q}_\hbar}^\hbar$. In other words, $\{f_{\hbar_n}\}_{n=1}^\infty$ is a minimizing sequence for the classical variational problem \eqref{classical VP}. Thus, it follows from Guo and Rein \cite{GuoRein1, Rein} (see \cite[Theorem 1.1]{Rein}) that there exist a minimizer
$$\mathcal{Q}=\tilde{\beta}\big(-\mu-(\tfrac{|p|^2}{2}-U(q))\big)$$
for the classical minimization problem \eqref{classical VP} and translation parameters $\{x_n\}_{n=1}^\infty\subset\mathbb{R}^3$ such that passing to a subsequence,
$f_{\hbar_n}(\cdot-x_n,\cdot)-\mathcal{Q}\rightharpoonup 0$ weakly in $L^\alpha(\mathbb{R}^6)$ for all $\alpha>\frac{5}{3}$, and 
\begin{equation}\label{U tilde convergence}
\nabla \tilde{U}_{\hbar_n}(\cdot-x_n)\to \nabla U\quad\textup{strongly in }L^2(\mathbb{R}^3),
\end{equation}
where $\tilde{U}_{\hbar}:=\frac{1}{|x|}*\rho_{f_\hbar}$. Therefore, combining \eqref{eq: CtoQ potential energy convergence} and \eqref{U tilde convergence}, we conclude that
\begin{equation}\label{gradient L^2 bound}
\nabla U_{\hbar_n}(\cdot-x_n)\to \nabla U\quad\textup{strongly in }L^2(\mathbb{R}^3).
\end{equation}
From now on, we abuse the notations by denoting $U_{\hbar_n}(\cdot-x_n)$ (resp., $f_{\hbar_n}(\cdot-x_n,\,\cdot)$) by $U_{\hbar_n}$ (resp., $f_{\hbar_n}$).

 Applying the Sobolev inequality to \eqref{gradient L^2 bound}, we obtain $\|U_{\hbar_n}-U\|_{L^6}\to 0$. On the other hand, by the Sobolev inequality, we have $\||\nabla|^{\frac{7}{10}}(U_{\hbar_n}-U)\|_{L^6}=\||\nabla|^{-\frac{13}{10}}(\rho_{\mathcal{Q}_{\hbar_n}}-\rho_{\mathcal{Q}})\|_{L^6}\lesssim \|\rho_{\mathcal{Q}_{\hbar_n}}-\rho_{\mathcal{Q}}\|_{L^\frac{5}{3}}\leq \|\rho_{\mathcal{Q}_{\hbar_n}}\|_{L^\frac{5}{3}}+\|\rho_{\mathcal{Q}}\|_{L^\frac{5}{3}}$. Hence, Lemma \ref{improved uniform boundedness} $(iii)$ and Proposition \ref{prop-Q} $(v)$ yield uniform boundedness of $\||\nabla|^{\frac{7}{10}}(U_{\hbar_n}-U)\|_{L^6}$. Thus, by interpolation, we obtain $\||\nabla|^s(U_{\hbar_n}-U)\|_{L^{6}}\to 0$ for all $0\leq s<\frac{7}{10}$. Therefore, by Morrey's inequality, we conclude that $U_{\hbar_n}\to U$ in $C^{0,\alpha}(\mathbb{R}^3)$ for all $0\leq\alpha<\frac{1}{5}$.

It remains to show that $\mu_{\hbar_n}\to\mu$. By Lemma \ref{improved uniform
 boundedness} $(iii)$ and Corollary \ref{Lagrange multiplier bound}, there exists a subsequence of $\{\hbar_n\}_{n=1}^\infty$ but still denoted by $\{\hbar_n\}_{n=1}^\infty$ such that
 $$\mu_{\hbar_n}\to \tilde{\mu}<0$$
for some $\tilde{\mu}>0$. We claim that $\tilde{\mu}=\mu$. Indeed, since $f_\hbar$ weakly converges to $\mathcal{Q}$ in $L^\alpha(\mathbb{R}^6)$, we have 
$$\int_{\mathbb{R}^6} \tilde{\beta}\big(\mu_{\hbar_n}-(\tfrac{|p|^2}{2}-U_{\hbar_n}(q))\big)g(q,p)\,dqdp\to \int_{\mathbb{R}^6} \tilde{\beta}\big(\mu-(\tfrac{|p|^2}{2}-U(q))\big) g(q,p)\, dqdp$$
for $g \in C^\infty_c(\mathbb{R}^6)$. On the other hand, by the Sobolev embedding, $U_{\hbar_n}$ converges to $U$ in $L^6(\mathbb{R}^3)$, and consequently it converges almost everywhere. We also recall from Proposition \ref{uniform boundedness} that $\|U_\hbar\|_{L^\infty}$ is uniformly bounded in $\hbar\in(0,1]$. 
We then apply the Lebesgue dominated convergence theorem to obtain
$$\int_{\mathbb{R}^6} \tilde{\beta}\big(\mu_{\hbar_n}-(\tfrac{|p|^2}{2}-U_{\hbar_n}(q))\big)g(q,p)\,dqdp\to \int_{\mathbb{R}^6} \tilde{\beta}\big(\tilde{\mu}-(\tfrac{|p|^2}{2}-U(q))\big)g(q,p) dqdp.$$
Since $g$ is arbitrarily chosen, this proves $\tilde{\mu} = \mu$.
\end{proof}

%$$\mathcal{Q}_\hbar\underset{\textup{Wigner+limit}}\longrightarrow\mathcal{Q}$$
%$$\mathcal{Q}_\hbar\underset{\textup{T\"oplitz+limit}}\longleftarrow\mathcal{Q}$$
\begin{proof}[Proof of Theorem \ref{thm1} $(iv)$]
We recall that $\mathcal{Q}_\hbar$ can be written by
$$\mathcal{Q}_\hbar = \tilde{\beta}\big(\mu_\hbar-(-\tfrac{\hbar^2}{2}\Delta-U_\hbar)\big).$$
By Theorem \ref{thm1}, there exist sequences $\{\hbar_n\}_{n=1}^\infty\subset(0,\infty)$, $\{x_n\}_{n=1}^\infty \subset \mathbb{R}^{3}$ and a minimizer $\mathcal{Q} = \tilde{\beta}(\mu-(\tfrac{|p|^2}{2}-U(q)))$ for $\mathcal{J}_{\min}$ such that $\hbar_n\to 0$, $\mu_{\hbar_n} \to \mu$ and $U_{\hbar_n}(\cdot-x_n) \to U$ strongly in $\dot{H}^1(\mathbb{R}^3) \cap L^\infty(\mathbb{R}^3)$ as $n\to \infty$.

Note that 
\begin{equation}\label{translated Q}
\tau_{-a}\mathcal{Q}_{\hbar}\tau_{a} = \tilde{\beta}\big(\mu_{\hbar}-(-\tfrac{\hbar^2}{2}\Delta-U_{\hbar}(\cdot-a))\big).
\end{equation}
To see this, we denote the negative eigenvalues and the corresponding normalized eigenfunctions of the Schr\"odinger $(-\tfrac{\hbar^2}{2}\Delta-U_{\hbar})$ by $\lambda_j$'s and $\phi_j$'s respectively. Then, we have
$$\tau_{-a}\mathcal{Q}_{\hbar}\tau_{a}=\tau_{-a}\bigg\{\sum_j\tilde{\beta}(\mu_\hbar-\lambda_j)\ket{\phi_j}\bra{\phi_j}\bigg\}\tau_{a}=\sum_j\tilde{\beta}(\mu_\hbar-\lambda_j)\ket{\phi_j(\cdot-a)}\bra{\phi_j(\cdot-a)}.$$
By simple translation, $(-\tfrac{\hbar^2}{2}\Delta-U_{\hbar})\phi_j=\lambda_j\phi_j$ if and only if $(-\tfrac{\hbar^2}{2}\Delta-U_{\hbar}(\cdot-a))\phi_j(\cdot-a)=\lambda_j\phi_j(\cdot-a)$. Thus, by functional calculus, \eqref{translated Q} follows. 

We define 
$$f_{\hbar_n} := \tilde{\beta}\big(\mu_{\hbar_n}-(\tfrac{|p|^2}{2}-U_{\hbar_n}(q-x_n))\big).$$
Then, Proposition \ref{Toplitz quant convergence} and \ref{TQ-bounded} yield
$$\begin{aligned}
\big\|\tau_{-x_n}\mathcal{Q}_{\hbar_n}\tau_{x_n} - \textup{Op}^T_{\hbar_n}[\mathcal{Q}]\big\| &\leq \big\|\tau_{-x_n}\mathcal{Q}_{\hbar_n}\tau_{x_n} - \textup{Op}^T_{\hbar_n}[f_{\hbar_n}]\big\|+\big\|\textup{Op}^T_{\hbar_n}[f_{\hbar_n}-\mathcal{Q}]\big\| \\
&\leq o_n(1) +\|f_{\hbar_n}-\mathcal{Q}\|_{L^\infty(\R^6)}.
\end{aligned}$$
Since $\tilde{\beta}$ is continuous and $\mu_{\hbar_n} \to \mu$, $U_{\hbar_n}(\cdot-x_n) \to U$ in $L^\infty$ as $n\to\infty$, we deduce that $\|f_{\hbar_n}-\mathcal{Q}\|_{L^\infty(\R^6)}\to 0$.

We now observe from Proposition \ref{WT-bounded} and $(ii)$ that
{\[
\begin{aligned}
&\big\|\tilde{W}_{\hbar_n}[\tau_{-x_n}\mathcal{Q}_{\hbar_n}\tau_{x_n}] - \mathcal{Q}\big\|_{L^\infty(\R^6)} \\
&\leq \big\|\tilde{W}_{\hbar_n}[\tau_{-x_n}\mathcal{Q}_{\hbar_n}\tau_{x_n}] - \tilde{W}_{\hbar_n}[\textup{Op}^T_{\hbar_n}[\mathcal{Q}]]\big\|_{L^\infty(\R^6)}+\big\|\tilde{W}_{\hbar_n}[\textup{Op}^T_{\hbar_n}[\mathcal{Q}]] -\mathcal{Q}\big\|_{L^\infty(\R^6)} \\
&\leq \big\|\tau_{-x_n}\mathcal{Q}_{\hbar_n}\tau_{x_n} -\textup{Op}^T_{\hbar_n}[\mathcal{Q}]\big\|+\big\|G^{6}_{\hbar/2}*G^{6}_{\hbar/2}*\mathcal{Q} -\mathcal{Q}\big\|_{L^\infty(\R^6)} \to 0.
\end{aligned}
\]}
This completes the proof. 

\end{proof}

\appendix

\section{T\"{o}plitz quantization and {Husimi} transforms}\label{appendix: Toplitz}

In this appendix, we collect useful properties on the T\"{o}plitz quantization and the {Husimi} transform. Recall that for a real-valued function $f$ on the phase space, its T\"{o}plitz quantization is defined as
$$\textup{Op}^T_\hbar[f] := \frac{1}{(2\pi\hbar)^3}\int_{\R^6} |\varphi^\hbar_{(q,p)}\rangle\langle \varphi^\hbar_{(q,p)}|f(q,p) dqdp,$$
where $\varphi^\hbar_{(q,p)}(x) = \frac{1}{(\pi\hbar)^{3/4}}e^{-\frac{|x-q|^2}{2\hbar}}e^{\frac{ip\cdot x}{\hbar}}$. By the T\"{o}plitz quantization, a function is transformed into an operator. It has the following mapping properties (see Appendix B of \cite{GMP}). 

\begin{proposition}\label{TQ-bounded}
Suppose that $f:\mathbb{R}^6\to\mathbb{R}$ is contained in $L^\infty(\R^{6})$. 
\begin{enumerate}[$(i)$]
\item $\textup{Op}^T_\hbar[f]$ is bounded and self-adjoint on $L^2(\mathbb{R}^3)$. Moreover, we have
$$\|\textup{Op}^T_\hbar[f]\| \leq \|f\|_{L^\infty(\mathbb{R}^3)}.$$
\item If $f$ is non-negative, so is $\textup{Op}^T_\hbar[f]$.
\item If we further assume that $f\in L^\alpha(\mathbb{R}^6)$ for some $1\leq \alpha<\infty$, then $\textup{Op}^T_\hbar[f]$ is a compact operator.
\end{enumerate}
\end{proposition}

A direct calculation shows that
\begin{equation}\label{density of quantization}
\rho^\hbar_{\textup{Op}^T_\hbar[f]}(x)=\int_{\R^3}\frac{1}{(\pi\hbar)^{3/2}}e^{-\frac{|x-q|^2}{\hbar}}\rho_f(q) dq=G_{\hbar/2}^3*\rho_f,
\end{equation}
where $G^d_a$ is the centered Gaussian function on $\R^d$ with the covariance matrix $aI$. Thus, the T\"oplitz quantization has almost the same density function as that of the given distribution function.

The following proposition asserts the T\"oplitz quantization  also almost preserve the mass and the kinetic and the potential energies.

\begin{proposition}\label{TQ-conserv} Suppose that $f:\mathbb{R}^6\to\mathbb{R}$ is non-negative.
\begin{enumerate}[$(i)$]
\item (Mass) If $f\in L^1(\mathbb{R}^6)$, then $\textup{Tr}^\hbar(\textup{Op}^T_\hbar[f])=\|f\|_{L^1(\mathbb{R}^6)}$. 
\item (Kinetic energy) If $f, |p|^2f\in L^1(\R^6)$, then
$$\lim_{\hbar\to0}\textup{Tr}^\hbar\big(-\hbar^2\Delta\textup{Op}^T_\hbar[f]\big)=\int_{\mathbb{R}^6} |p|^2f(q,p)dqdp.$$
\item (Potential energy) If $\rho_f \in L^{\frac{6}{5}}(\R^3)$, then
$$\lim_{\hbar\to0}\int_{\R^6}\frac{\rho^{\hbar}_{\textup{Op}^T_{\hbar}[f]}(x)\rho^{\hbar}_{\textup{Op}^T_{\hbar}[f]}(y)}{|x-y|}dxdy  \to \int_{\R^6}\frac{\rho_f(x)\rho_f(y)}{|x-y|}dxdy.$$
\end{enumerate}
\end{proposition}

\begin{proof} $(i)$ is obvious. For $(ii)$, direct calculations yield
$$\begin{aligned}
\textup{Tr}^\hbar\big(-\hbar^2\Delta\textup{Op}^T_\hbar[f]\big)&=\sum_{j=1}^3\int_{\R^6} \textup{Tr}\big(|\hbar\partial_{x_j}\varphi^\hbar_{(q,p)}\rangle\langle \hbar\partial_{x_j}\varphi^\hbar_{(q,p)}|\big)f(q,p) dqdp\\
&=\sum_{j=1}^3\int_{\R^6} \left\{\int_{\mathbb{R}^3} \big(p_j^2+(x_j-q_j)^2\big)|\varphi^\hbar_{(q,p)}(x)|^2dx\right\}f(q,p) dqdp\\
&\to \int_{\mathbb{R}^6} |p|^2f(q,p)dqdp.
\end{aligned}$$
$(iii)$ follows from the Hardy-Littlewood-Sobolev inequality and \eqref{density of quantization}. 
\end{proof}

The T\"oplitz quantization also nearly preserves products. 
\begin{proposition}\label{TQ-est2}
If $f, g \in L^\infty(\R^6)\cap L^\alpha(\mathbb{R}^6)$ for some $1\leq\alpha<\infty$, then
$$\lim_{\hbar\to0}\big\|\textup{Op}^T_\hbar[fg]-\textup{Op}^T_\hbar[f]\textup{Op}^T_\hbar[g]\big\|=0.$$
%\item If $f, g \geq 0$ and $\nabla g \in L^1(\R^d)$, then
%\[
%|\textup{Tr}^\hbar(\textup{Op}^T[fg])-\textup{Tr}^\hbar(\textup{Op}^T[f]\textup{Op}^T[g])| \leq C\sqrt{\hbar}\|f\|_{L^\infty}\|\nabla g\|_{L^1}
%\]
%\end{enumerate}
\end{proposition}

\begin{proof}
By a standard density argument with Proposition \ref{TQ-bounded} $(i)$, we may assume that $f,g\in C_c^\infty(\mathbb{R}^6)$. By the definition, we have
$$\textup{Op}^T_\hbar[f]\textup{Op}^T_\hbar[g]= \frac{1}{(2\pi\hbar)^6}\int_{\mathbb{R}^{12}} f(q,p)g(q',p')|\varphi^\hbar_{(q,p)}\rangle\langle \varphi^\hbar_{(q,p)}| \varphi^\hbar_{(q',p')}\rangle \langle \varphi^\hbar_{(q',p')}|dq'dp'dqdp,$$
while inserting
$$\frac{1}{(2\pi\hbar)^3}\int_{\mathbb{R}^6}|\varphi_{(q,p)}^\hbar\rangle\langle \varphi_{(q,p)}^\hbar| dqdp=\textup{Id}_{L^2(\mathbb{R}^3)},$$
we may write 
$$\textup{Op}^T_\hbar[fg] =\frac{1}{(2\pi\hbar)^6}\int_{\mathbb{R}^{12}}  f(q,p)g(q,p)|\varphi^\hbar_{(q,p)}\rangle\langle \varphi^\hbar_{(q,p)}| \varphi^\hbar_{(q',p')}\rangle \langle \varphi^\hbar_{(q',p')}|dq'dp'dqdp. $$
A direct computation shows
$$\langle \varphi^\hbar_{(q,p)}|\varphi^\hbar_{(q',p')}\rangle = 2^{\frac{3}{2}}e^{-\frac{|q-q'|^2}{4\hbar}}e^{-\frac{|p-p'|^2}{4\hbar}}e^{-\frac{i(p-p')(q+q')}{2\hbar}}.$$
Hence, for any $\psi , \tilde\psi \in L^2(\R^3)$ with $\|\psi\|_{L^2} = \|\tilde\psi\|_{L^2} = 1$, we have
$$\begin{aligned}
&\big|\big\langle\tilde{\psi}\big|(\textup{Op}^T_\hbar[fg]-\textup{Op}^T_\hbar[f]\textup{Op}^T_\hbar[g])\big|\psi\big\rangle\big|   \\
&\lesssim \frac{1}{\hbar^{6}}\int_{\mathbb{R}^{12}} e^{\frac{-|p-p'|^2}{4\hbar}}e^{\frac{-|q-q'|^2}{4\hbar}}|f(q,p)|\big|g(q',p')-g(q,p)\big|\big|\langle\tilde\psi| \varphi^\hbar_{(q,p)}\rangle\big|\big|\langle \varphi^\hbar_{(q',p')}|\psi\rangle\big|dq'dp'dqdp\\
%& \lesssim \frac{1}{\hbar^{2d}}\left(\iiiint |g(q',p')-g(q,p)|e^{-\frac{|q-q'|^2}{4\hbar}}e^{\frac{-|p-p'|^2}{4\hbar}}\right. \\
%&\qquad\qquad\qquad\qquad\qquad\qquad\left.\times|\langle \varphi^\hbar_{(q,p)} ,\tilde\psi\rangle||\langle \phi^\hbar_{(q',p')}, \psi\rangle|\,dq'dp'dqdp\right)\|f\|_{L^\infty} \\
%& \lesssim \frac{1}{\hbar^{2d}}\left(\iiiint |(q'\!-q, p'\!-p)|^{\alpha}e^{\frac{-|q-q'|^2}{4\hbar}}e^{\frac{-|p-p'|^2}{4\hbar}}\right. \\
%&\qquad\qquad\qquad\qquad\qquad\qquad\left.\times|\langle \varphi^\hbar_{(q,p)} ,\tilde\psi\rangle||\langle \phi^\hbar_{(q',p')}, \psi\rangle| \,dq'dp'dqdp\right) \|f\|_{L^\infty}\|g\|_{C^{0,\alpha}} \\
& \lesssim \frac{1}{\hbar^{6}}\int_{\mathbb{R}^{12}} e^{\frac{-|q-q'|^2}{4\hbar}}e^{\frac{-|p-p'|^2}{4\hbar}}\|f\|_{L^\infty}\|g\|_{C^1}|(q'-q, p'-p)|\big|\langle\tilde\psi| \varphi^\hbar_{(q,p)}\rangle\big|\big|\langle \varphi^\hbar_{(q',p')}|\psi\rangle\big|dq'dp'dqdp \\
& \lesssim \frac{\sqrt{\hbar}}{\hbar^{6}}\bigg\{\int_{\mathbb{R}^{12}}e^{\frac{-|q-q'|^2}{4\hbar}}e^{\frac{-|p-p'|^2}{4\hbar}}\big|\langle\tilde\psi| \varphi^\hbar_{(q,p)}\rangle\big|\big|\langle \varphi^\hbar_{(q',p')}|\psi\rangle\big|dq'dp'dqdp\bigg\}\|f\|_{L^\infty}\|g\|_{C^1}\\
& \sim \frac{\sqrt{\hbar}}{\hbar^3}\bigg\{\int_{\mathbb{R}^{6}}\big|\langle\tilde\psi| \varphi^\hbar_{(q,p)}\rangle\big|\big(G_{2\hbar}^6*\big|\langle \varphi^\hbar_{(\cdot,\cdot)}|\psi\rangle\big|\big)(q,p)dqdp\bigg\}\|f\|_{L^\infty}\|g\|_{C^1}\\
& \leq \frac{\sqrt{\hbar}}{\hbar^3}\big\|\langle\tilde\psi| \varphi^\hbar_{(q,p)}\rangle\big\|_{L_{q,p}^2}\big\|\langle \varphi^\hbar_{(q,p)}|\psi\rangle\big\|_{L_{q,p}^2}\|f\|_{L^\infty}\|g\|_{C^1}.
\end{aligned}$$
Note that since
$$\langle \varphi^\hbar_{(q,p)}|\psi\rangle=\int_{\mathbb{R}^3}\frac{1}{(\pi\hbar)^{3/4}}e^{-\frac{|x-q|^2}{2\hbar}}e^{-\frac{ip\cdot x}{\hbar}}\psi(x)dx=\mathcal{F}_x\Big(\frac{1}{(\pi\hbar)^{3/4}}e^{-\frac{|\cdot-q|^2}{2\hbar}}\psi\Big)(\tfrac{p}{\hbar}),$$
where $\mathcal{F}_x$ denotes the Fourier transform, it follows from the Plancherel theorem that
$$\begin{aligned}
\int_{\mathbb{R}^6} |\langle \varphi^\hbar_{(q,p)}|\psi\rangle|^2dqdp&=\int_{\mathbb{R}^6} \Big|\mathcal{F}_x\Big(\frac{1}{(\pi\hbar)^{3/4}}e^{-\frac{|\cdot-q|^2}{2\hbar}}\psi\Big)(\tfrac{p}{\hbar})\Big|^2dqdp\\
&\sim \int_{\mathbb{R}^6} \Big|\mathcal{F}_x\Big(e^{-\frac{|\cdot-q|^2}{2\hbar}}\psi\Big)(p)\Big|^2dqdp\sim \int_{\mathbb{R}^6} \big|e^{-\frac{|x-q|^2}{2\hbar}}\psi(x)\big|^2dxdq\sim \hbar^{\frac{3}{2}}.
\end{aligned}$$
Therefore, inserting this, we complete the proof.
\end{proof}

The Wigner transform of an operator $\gamma \in \mathcal{B}(L^2)$ with kernel $\gamma(x,x')$ is given by
\[
W_\hbar[\gamma](q,p) := \int_{\R^3}\gamma\big(q+\tfrac{y}{2}, q-\tfrac{y}{2}\big)e^{-\frac{i p \cdot y}{\hbar}}dy.
\]
The following proposition shows that the Wigner transform is asymptotically an inverse of the T\"oplitz operator. 
See \cite{GMP}. 
{\begin{proposition}
The following holds.
\[
W_\hbar[\textup{Op}^T_\hbar[f]] = G^{6}_{\hbar/2}*f,
\]
where $G^n_a$ is the centered Gaussian density on $\R^n$ with covariance matrix $aI$.
\end{proposition}}

{
The Husimi transform $\tilde{W}_{\hbar}[\gamma]$ of an operator $\gamma$ is given by
\[
\tilde{W}_{\hbar}[\gamma] \coloneqq G^6_{\hbar/2}*W_\hbar[\gamma].
\] 
Due to the regularizing effect of the Gaussian kernel  $G^6_{\hbar/2}$, the Husimi transform enjoys the following $L^\infty$ estimate.
\begin{proposition}\label{WT-bounded}
Let $\gamma$ be in $\mathcal{B}(L^2)$. Then there holds
\[
\|\tilde{W}_\hbar[\gamma]\|_{L^\infty(\R^6)} \leq \|\gamma\|_{\mathcal{B}(L^2)}.
\]
\end{proposition}
\begin{proof}
By duality, one has
\[
\|\tilde{W}_\hbar[\gamma]\|_{L^\infty} = \sup \left\{ \int \tilde{W}_\hbar[\gamma]f\,dqdp ~\Big|~ f \in L^1(\R^6)  \right\}
\]
The equality (54) in \cite{GMP} writes for $f \in L^1(\R^6)$,
\[
\text{Tr}^\hbar[\text{OP}^T_\hbar[f]\gamma] = \int \tilde{W}_\hbar[\gamma]f\,dqdp
\]
so that by Proposition \ref{TQ-conserv}
\[
\int \tilde{W}_\hbar[\gamma]f\,dqdp \leq \|\gamma\|_{\mathcal{B}(L^2)}\left|\text{Tr}^\hbar[\text{OP}^T_\hbar[f]]\right|
=  \|\gamma\|_{\mathcal{B}(L^2)}\|f\|_{L^1}.
\]
This proves the proposition.
\end{proof}
}

\section{Thermal effects for the gravitational Hartree equation in semi-classical limit}\label{sec: thermal effects}

We give a precise statement on thermal effects for the gravitational Hartree equation in  Aki-Dolbeault-Sparber \cite{ADS}, and we present the connection to our results. To highlight the role of temperature, separating $\beta(s)=T\beta_0(s)$ with $T>0$, we introduce the free energy of the form 
$$\mathcal{J}_T^{\hbar}(\gamma):=\mathcal{E}^\hbar(\gamma)+T\mathcal{C}_0^\hbar(\gamma),$$
where $\mathcal{C}_0^\hbar(\gamma)=\textup{Tr}^\hbar(\beta_0(\gamma))$, and consider the corresponding minimization problem
\begin{equation}\label{VP with T}
\mathcal{J}^{\hbar}_{T, \min}:=\min_{\gamma \in \mathcal{A}_M^\hbar}\mathcal{J}^{\hbar}_T(\gamma).
\end{equation}
Then, we define the \textit{maximal temperature} by 
$$T^{\hbar}_* := \sup\Big\{T > 0 ~|~ \mathcal{J}^{\hbar}_{T, \min} < 0\Big\}$$
and the \textit{critical temperature} by
$$T^{\hbar}_c := \sup\Big\{T> 0 ~|~ \mathcal{J}^{\hbar}_{T, \min} = \mathcal{J}^{\hbar}_{0,\min}+\tau \beta_0(M),\quad \forall \tau \in (0,T]\Big\}.$$
\begin{theorem}[Aki-Dolbeault-Sparber \cite{ADS}]\label{result-ADS with T}
Let $M>0$. If $\beta_0$ satisfies \textup{\textbf{(A1)}} and \textup{\textbf{(A4)}}, then the maximal temperature $T^{\hbar}_*$ is positive (possibly infinite), and the following statements hold:
\begin{enumerate}[$(i)$]
\item If $T<T_*^\hbar$, then the minimization problem \eqref{VP with T} possesses a minimizer. A minimizer $\mathcal{Q}_\hbar$ solves the self-consistent equation
\begin{equation*}
\mathcal{Q}_\hbar =\tilde{\beta}_0\Big(\frac{-\mu_\hbar-\hat{E}_\hbar}{T}\Big),\quad\textup{with some }\mu_\hbar>0,
\end{equation*}
where $\hat{E}_\hbar=-\frac{\hbar^2}{2}\Delta-U_\hbar$ is the quantum mean-field Hamiltonian with $U_\hbar=\frac{1}{|x|}*\rho_{\mathcal{Q}_\hbar}^\hbar$.
\item The critical temperature $T^\hbar_c$ satisfies $0 < T^\hbar_c < T^\hbar_*$, and a minimizer is a pure state if and only if $T \in [0, T^\hbar_c]$.
\end{enumerate}
\end{theorem}

\begin{remark} $(i)$ The main result in \cite{ADS} is stated for $\hbar = 1$, but it can be extended directly to general but fixed $\hbar > 0$. \\
$(ii)$ If $T^{\hbar}_*<\infty$, a minimizer may not exist at too high temperature. Theorem \ref{result-ADS with T} $(iii)$ shows that Bose-Einstein condensation can be observed at sufficiently low (not necessarily zero) temperature, which is physically relevant. Possibility of non-existence of a minimizer and existence of pure states are completely quantum mechanical phenomena.
\end{remark}

Combining with some of results in this paper, we show that the quantum thermal effects vanish in the semi-classical limit $\hbar\to0$ in the following sense.

\begin{proposition}
As $\hbar>0$ goes to zero, the maximal temperature $T^{\hbar}_*$ goes to infinity, while the critical temperature $T_c^\hbar$ converges to zero.
\end{proposition}

\begin{proof}
Fix any $T>0$, and set $\beta(s)=T\beta_0(s)$, where $\beta_0$ is a function satisfying $\textup{\textbf{(A1)}}$ - $\textup{\textbf{(A4)}}$ (so is $\beta$). Then, by Proposition \ref{prop: CtoQ}, the minimum free energy $\mathcal{J}^{\hbar}_{T, \min}$ is negative, in other word, $T_*^\hbar\geq T$, for all sufficiently small $\hbar>0$. On the other hand, Weyl's law (Proposition \ref{Weyl's law for profile}) implies that the number of negative eigenvalues of $\hat{E}_\hbar=-\frac{\hbar^2}{2}\Delta-U_\hbar$ goes to infinity. Hence, if $\hbar$ is small enough, $\mathcal{Q}_\hbar$ is not a pure state, and by Proposition \ref{result-ADS with T} $(ii)$, we have $T_c^\hbar\leq T$. Since $T$ is arbitrary, we conclude that $T_*^\hbar\to \infty$ and $T_c^\hbar\to 0$.
\end{proof}

\end{document}